\documentclass[11pt, psamsfonts, reqno]{amsart}

\usepackage[T1]{fontenc}
\usepackage[latin1]{inputenc}
\usepackage[frenchb, english]{babel}
\usepackage{amsthm}
\usepackage{amssymb,amsmath,MnSymbol,bbm}
\usepackage[all]{xy}
\usepackage{mathrsfs}
\usepackage[left=0.0cm,right=0.0cm,top=1.7cm,bottom=1.7cm]{geometry}
\usepackage{hyperref}
\usepackage{verbatim}

\theoremstyle{plain}
\newtheorem{theorem}{Theorem}[section]

\newtheorem{lemma}[theorem]{Lemma}
\newtheorem{proposition}[theorem]{Proposition}

\newtheorem{corollary}[theorem]{Corollary}

\newtheorem*{theorem*}{Theorem}

\theoremstyle{definition}
\newtheorem{definition}[theorem]{Definition}

\theoremstyle{remark}
\newtheorem{remark}[theorem]{Remark}

\usepackage[usenames,dvipsnames]{color}

\usepackage{tikz}
\usetikzlibrary{calc}

\def\Q{{\bf Q}}

\def\Z{{\bf Z}}
\def\C{{\bf C}}
\def\N{{\bf N}}
\def\R{{\bf R}}

\def\O{{\mathcal{O}}}
\def\H{{H}}

\def\SL{{\mathbf{SL}}}

\def\zp{{\Z_p}}

\def\zpe{{\mathbf{Z}_p^\times}}

\def\qp{{\Q_p}}

\def\A{{\bf A}}
\def\Af{{\bf A}_{f}}
\def\EE{{\mathcal{E}}}

\def\D{{\bf D}}

\def\Dcris{{\D_{\mathrm{crys}}}}

\def\Bcris{{\mathbf{B}_{\mathrm{cris}}}}

\def\epsilon{\varepsilon}
\def\det{\mathrm{det}}
\def\Sh{{\operatorname{Sh}}}
\def\br{{\operatorname{br}}}
\def\et{{\rm \acute{e}t}}
\def\GSp{{\mathbf{GSp}}}
\def\Sp{{\mathbf{Sp}}}
\def\GL{\mathbf{GL}}
\def\Gm{\mathbf{G}_{\rm m}}
\def\G{\mathbf{G}}
\def\H{\mathbf{H}}

\def\matrix#1#2#3#4{{\big(\begin{smallmatrix}#1&#2\\ #3&#4\end{smallmatrix}\big)}}

\title{Norm-compatible systems of Galois cohomology classes for $\GSp_{6}$}

\author{Antonio Cauchi}
\thanks{* The first author was supported by the Engineering and Physical Sciences Research Council [EP/L015234/1]. The EPSRC Centre for Doctoral Training in Geometry and Number Theory (The London School of Geometry and Number Theory), University College London.}

\author{Joaqu\'in Rodrigues Jacinto}
\thanks{* The second author was supported by Sarah Zerbes' ERC Consolidator Grant \textit{Euler systems and the Birch-Swinnerton-Dyer conjecture}.}

\makeatletter
\def\@tocline#1#2#3#4#5#6#7{\relax
  \ifnum #1>\c@tocdepth % then omit
  \else
    \par \addpenalty\@secpenalty\addvspace{#2}%
    \begingroup \hyphenpenalty\@M
    \@ifempty{#4}{%
      \@tempdima\csname r@tocindent\number#1\endcsname\relax
    }{%
      \@tempdima#4\relax
    }%
    \parindent\z@ \leftskip#3\relax \advance\leftskip\@tempdima\relax
    \rightskip\@pnumwidth plus4em \parfillskip-\@pnumwidth
    #5\leavevmode\hskip-\@tempdima
      \ifcase #1
       \or\or \hskip 1em \or \hskip 2em \else \hskip 3em \fi%
      #6\nobreak\relax
    \dotfill\hbox to\@pnumwidth{\@tocpagenum{#7}}\par
    \nobreak
    \endgroup
  \fi}
\makeatother
\usepackage{hyperref}

\begin{document}
\maketitle
\selectlanguage{english}

\begin{abstract} We construct global cohomology classes in the middle degree cohomology of the Shimura variety of the symplectic group $\GSp_6$ compatible when one varies the level at $p$. These classes are expected constituents of an Euler system for the Galois representations appearing in these cohomology groups. As an application, we show how these classes provide elements in the Iwasawa cohomology of these representations and, by applying Perrin-Riou's machinery, $p$-adic L-functions associated to them. \end{abstract}

\setcounter{tocdepth}{2}
\tableofcontents
%\newpage
\section{Introduction}

The construction of special elements in the motivic and \'etale cohomology of Shimura varieties has contributed to proving cases of Beilinson's  conjectures on special values of motivic L-functions (e.g. \cite{Beilinson2}, \cite{kings98}, \cite{lemmarf} etc.), the Birch and Swinnerton-Dyer conjecture and the Bloch-Kato conjecture (e.g. \cite{Kato}, \cite{BDR2}, \cite{KLZ}, \cite{LSZ1} etc.), and it constitutes one of the main tools to study the arithmetic of Galois representations appearing in the cohomology of Shimura varieties and their relation to special L-values. \\

Fix a prime number $p$. In this article, we construct elements in the cohomology of the Shimura variety of the symplectic similitude group $\G=\GSp_6 / \Q$, which satisfy the Euler system norm relations in the cyclotomic tower at $p$. As an application, we show how they give rise to elements in the Iwasawa cohomology of Galois representations appearing in the middle degree $p$-adic \'etale cohomology of the Shimura variety for $\G$. 
The strategy we adopt here has been inspired by the work of \cite{LLZ1} on the construction of the Beilinson-Flach Euler system, which has also successfully been applied in many different contexts (\cite{LLZ2}, \cite{LZsym},  \cite{LSZ1} etc.). 

\subsection{Setting}

We consider the subgroup $\H=\GL_2 \times_\det \GL_2 \times_\det \GL_2 \subset \G$, which, after a suitable choice of Shimura datum, induces an embedding $\iota:\Sh_{\H} = \Sh(\H,X_{\H})\hookrightarrow \Sh(\G,X_{\G}) = \Sh_{\G}$. By pulling back Beilinson's Eisenstein symbol in the motivic cohomology of the modular curve associated to the first $\GL_2$-copy of $\H$, we get elements in the first motivic cohomology group of $\Sh_{\H}$. Their push-forward along $\iota$ thus gives elements in the seventh motivic cohomology group of $\Sh_{\G}$. One then uses the natural action of $\G(\Af)$ on the Shimura variety  $\Sh_{\G}$ to perturb these classes and obtain a whole compatible system of cohomology classes defined over ramified extensions of the base field. Our setting is very similar to the one first considered in \cite{lemmanc} and later developed in \cite{LSZ1}.

\subsection{Motivation}

Let $\pi$ be a cohomological cuspidal automorphic  representation of $\G(\Af)$. After projecting to the $\pi$-isotypic component, the motivic classes that we construct are expected, according to Beilinson's conjectures, to be related to special values of the degree eight spin L-function $\mathrm{L}_{\operatorname{Spin}}(\pi, s )$ associated to $\pi$. This is motivated by recent work of Pollack and Shah (\cite{PollackShah}), who have given (under certain hypotheses on $\pi$) an integral representation of the (partial) spin L-function of $\pi$, by integrating over $\H$ a $\GL_2$-Eisenstein series against a cusp form $\varphi$ in the space of $\pi$.

\subsection{Main results}

After applying the \'etale regulator map and employing the action of the Hecke algebra of $\G$, we prove the following.

\begin{theorem*} \label{maingsp6} Let $\mathcal{W}$ be the $\zp$-local system associated to an arbitrary irreducible algebraic representation of $\G$.  There exists a family  of \'etale cohomology classes \[ z_{n,m}^{ \mathcal{W}} \in H^{7}_{\et}(\Sh_{\G}(K_{n,0})_{/{\Q(\zeta_{p^m})}}, \mathcal{W}(4+q)),\] for an appropriate integer $q$, which satisfies the following norm relations:
\begin{enumerate}
\item For $n \geq 1$, $(\operatorname{pr}^{K_{n+1,0}}_{K_{n,0}})_*( z_{n+1,m}^{ \mathcal{W}}) = z_{n,m}^{ \mathcal{W}}$;
 \item For $n,m \geq 1$, $\operatorname{norm}^{{\Q(\zeta_{p^{m+1}})}}_{{\Q(\zeta_{p^m})}}( z_{n,m+1}^{ \mathcal{W}}) =\tfrac{\mathcal{U}_p'}{\sigma_p^3} \cdot z_{n,m}^{ \mathcal{W}}$,
\end{enumerate} 
where $\mathcal{U}_p'$ is the Hecke operator associated to the double coset of $diag(p^{-3},p^{-2},p^{-2},p^{-1},p^{-1},1)  \in \G(\Q_p)\subset \G(\Af)$, and $\sigma_p$ is the image of $p^{-1}$ under the Artin map $\Q_p^* \hookrightarrow \Af^* \to Gal(\Q(\zeta_{p^m})/\Q)$.
\end{theorem*}

Some words about the theorem above. By $K_{n,0}$ we mean a tower of sufficiently small level subgroups of $\G(\widehat{\Z})$ defined by certain congruences modulo powers of $p$ (cf. \S \ref{levelgroups} for precise definitions). As is mentioned below, the first relation will allow us to vary our classes in families, while the second relation allows one to modify the classes $z_{n,m}^{ \mathcal{W}}$ so that they satisfy the usual Euler systems relation at $p$. Finally, it is essential for Iwasawa theoretic purposes to work with integral coefficients, which renders the construction of the classes more delicate. \\ 

By using the theory of $\Lambda$-adic Eisenstein classes developed in \cite{kings}, we also show that these classes vary $p$-adically in families thus obtaining a universal class interpolating them all. Taking specialisations of this universal class, one obtains more \'etale cohomology classes which do not a priori come from a geometric construction. \\

Let us briefly mention some immediate applications of our results. Using results from \cite{MokraneTilouine}, one can project our classes to the groups \[H^1(\Q(\zeta_{p^m}),V_{\pi}),\] where $\pi$ is a suitable automorphic cuspidal representation of $\G(\Af)$ and $V_{\pi}$ is, up to a twist of the cyclotomic character, the $p$-adic spin Galois representation associated to $\pi$. After imposing a $\mathcal{U}_p'$-ordinarity condition on $V_{\pi}$, one can slightly modify the classes above constructed to define Galois cohomology classes in the Iwasawa cohomology of $V_{\pi}$. Applying the general machinery of Perrin-Riou, this allows to define a $p$-adic spin L-function for this automorphic representation.

\subsection{Towards new Euler Systems}

We finally mention that this work should be seen as a first step towards constructing Euler systems for $\G$. The gap between our Galois cohomology classes and an Euler system is the absence of the so-called tame norm relations, which compare classes over fields $\Q(\zeta_{m\ell})$ and $\Q(\zeta_{m})$, where $\ell$ does not divide $m$. In \cite{LSZ1}, the authors introduce a technique for proving the tame norm relations, which relies on the local Gan-Gross-Prasad conjecture for the pair $(\mathbf{SO}_4,\mathbf{SO}_5)$. Similar techniques have also been used by C. Cornut and in \cite{Jetchev}. We hope to be able to come back to this in the future. \\

%At the moment, it seems very hard to adapt these techniques to our setting. It also seems to be very difficult to show that our classes are non-zero. \\
 
There are still many other natural questions yet to be answered concerning the classes constructed in this article. The relation between the special values of the $p$-adic spin L-function and the complex spin L-function are still mysterious. We expect an explicit reciprocity law to hold, relating values of Bloch-Kato's dual exponential maps of our Iwasawa class to certain values of the complex spin L-function, which should also show the non-vanishing of the classes. One should also be able to calculate the complex regulator of the motivic classes in terms of the complex spin L-function using the techniques of \cite{kings98} and \cite{lemmarf}. We are at the moment working on some of these points, as well as on generalizations to $\GSp_{2n}$ with $n > 3$, and we expect this work to be the first one of a series devoted to the study of the arithmetic of automorphic forms for symplectic groups.

\subsection{Acknowledgements}

 We would like to thank David Loeffler and Sarah Zerbes for having suggested the project to us and for their several valuable comments on it. Special thanks to Chris Skinner, who kindly explained to us various aspects of the theory. 
Parts of this work were conducted while both authors were visiting the Bernoulli Center at EPFL. We are grateful to the Bernoulli Center and specially to Dimitar Jetchev for their hospitality and invitation to participate in part of the semester "Euler Systems and Special Values of L-functions".

\section{Preliminaries}
\subsection{Groups}

Let
\[ \H = \GL_2 \times_\det \GL_2 \times_\det \GL_2 = \{ (A, B, C) \; : \; A, B, C \in \GL_2, \det A = \det B = \det C \} \] be the group scheme over $\Z$ obtained by taking the product over the determinant of three copies of $\GL_2$, and let $\G$ be the group scheme over $\Z$ defined by having $R$-points
\[ \G(R) = \GSp_6(R) = \{ A \in \GL_6(R) \; : \; A^t J A = \nu(A) J, \; \nu(A) \in \Gm(R) \}, \] for any commutative ring $R$ with 1, where we have fixed $J$ to be the matrix ${ \matrix 0 {I_3'} {-I_3'} 0}$, for $I_3'=\left( \begin{smallmatrix} & & 1 \\ & 1 & \\ 1 & & \end{smallmatrix} \right)$. In the following, we will consider $\H$ as a subgroup of $\G$ through the embedding defined by
\[ \Delta \colon \left( { \matrix {a_1} {b_1} {c_1} {d_1} },  {\matrix {a_2} {b_2} {c_2} {d_2} },  {\matrix {a_3} {b_3} {c_3} {d_3} } \right) \in \H \mapsto { \left( \begin{smallmatrix} a_1 & {} & {} & & {} &  b_1 \\ & a_2 & & & b_2 &   \\ & & a_3 &  b_ 3 & & \\  & {} & c_3& d_3 & {} & {} \\ & c_2 & & & d_2  \\ c_1 & &  & & & d_ 1  \end{smallmatrix} \right)} \in \G \]
We denote by $Z_\H$ and $Z_\G$ the centers of $\H$ and $\G$ respectively.

\subsection{Shimura varieties}  

Let $\mathbb{S} = \mathrm{Res}_{\C / \R} \mathbb{G}_m$ and define $h_{\GL_2} \colon \mathbb{S} \to {\GL_2}_{/ \R}$ by \[ h_{\GL_2}(a + i b) = \frac{1}{a^2+b^2}{\matrix a b {-b} a}\] on the real points. Let $X_{\GL_2}$ be the set of $\GL_2(\R)$-conjugacy classes of $h_{\GL_2}$. This induces a Shimura datum $(\GL_2, X_{\GL_2})$, and we denote by  $\Sh_{\GL_2}$ the Shimura variety thus defined, i.e. the usual modular curve. For $U$ a sufficiently small open compact subgroup of $\GL_2(\Af)$ \footnote{Recall that, for any Shimura datum $(G, X_G)$, a compact open subgroup $U \subseteq G(\Af)$ is said to be sufficiently small if it acts faithfully on $\Sh_G(\C) := G(\Q) \backslash G(\Af) \times X_G$. For any such $U$, $G(\Q) \backslash G(\Af) \times X_G / U$ is the set of complex points of an algebraic variety $\Sh_G(U)$, which is defined over a number field $E = E(G, X_G)$ called the reflex field of $(G, X_G)$.}, we denote by $\Sh_{\GL_2}(U)$ the varieties of corresponding level, with reflex field $\Q$ and whose complex points are given by
\[ \Sh_{\GL_2}(U)(\C) = \GL_2(\Q) \backslash X_{\GL_2} \times \GL_2(\Af) / U. \]

The diagonal embedding $\GL_2 \to \H$ induces a Shimura datum $(\H, X_{\H})$ and denote by $\Sh_{\H}$ the corresponding Shimura variety. Its reflex field is again $\Q$. If $U \subseteq \H(\Af)$ is a fibre product $U_1 \times_{det} U_2 \times_{det} U_3$ of (sufficiently small) subgroups of $\GL_2(\Af)$, we have \[ \Sh_{\H}(U) = \Sh_{\GL_2}(U_1) \times_{\Gm} \Sh_{\GL_2}(U_2) \times_{\Gm} \Sh_{\GL_2}(U_3), \]
where $\times_{\Gm}$ denotes the fibre product over the zero dimensional Shimura variety of level $D=det(U_i)$ \[\pi_0(\Sh_{\GL_2})(D)=\hat{\Z}^*/D\] given by the connected components of $\Sh_{\GL_2}$.
Finally, the embedding $\Delta$ induces another Shimura datum $(\G, X_{\G})$, with corresponding Shimura varieties $\Sh_{\G}$ with reflex field $\Q$. For sufficiently small $U \subseteq \G(\Af)$, $\Sh_{\G}(U)$ is a smooth quasi-projective scheme over $\Q$ whose complex points are given by 
\[ \Sh_{\G}(U)(\C) = \G(\Q) \backslash X_{\G} \times \G(\Af) / U. \]
We have an embedding $\Sh_{\H} \hookrightarrow \Sh_{\G} $ of codimension $3$, which for any open compact subgroup $U$ of $\G(\Af)$ gives
\[ \iota_U: \Sh_{\H}(U \cap \H) \longrightarrow \Sh_{\G}(U). \]

The following lemma is an adaptation of \cite[Lemma 5.3.1]{LSZ1}.

\begin{lemma}\label{closedimm}
Let $U$ be an open compact subgroup of $\G(\A_f)$ such that there exists a sufficiently small open compact subgroup $U'$ of $\G(\Af)$ containing $U$, $w_1 U w_1$ and $w_2 U w_2$, where $w_1=diag(-1,1,1,1,1,-1)$ and $w_2=diag(1,-1,1,1,-1,1)$. Then the morphism (of $\Q$-schemes) \[ \iota_U:\Sh_{\H}(U \cap \H) \longrightarrow \Sh_{\G}(U) \] is a closed immersion.
\end{lemma}

\begin{proof}
We note that it is enough to show it on the complex points of the Shimura varieties. As it was pointed out before, the map at infinite level $\Sh_{\H}(\C) \longrightarrow \Sh_{\G}(\C)$ is an injection, hence we need to show that if $z,z' \in \Sh_{\H}(\C)$ have the same image in $\Sh_{\G}(U)(\C)$, then $z=z'u$ for $u \in U \cap \H$. This would follow by showing that for any $u \in U \setminus (U \cap \H)$, we have $\Sh_{\H}(\C) \cap \Sh_{\H}(\C)u = \emptyset$ as subsets of $\Sh_{\G}(\C)$. 

We show the latter as follows. The quotient $W=Z_{\H} / (\H \cap Z_{\G})$ is generated by the two involutions $w_1$ and $w_2$. An easy calculation shows that the centraliser  $C_{\G(\A_f)}(\{ w_1,w_2 \})$ is $\H(\A_f)$. Note that the action of $w_1$ and $w_2$ on $\Sh_{\G}(\C)$ fixes $\Sh_{\H}(\C)$ pointwise. Thus, if $z,z u \in \Sh_{\H}(\C)$ for $u \in U$, the elements $v_1=u(w_1u^{-1}w_1)$  and $v_2=u(w_2u^{-1}w_2)$ fix $z$. By hypothesis $v_1,v_2 \in U'$, which acts faithfully on $\Sh_{\G}(\C)$, thus we conclude that $v_1 = v_2 = 1$. This implies that $u$ centralizes the subgroup generated by $w_1$ and $w_2$ and hence $u \in U \cap \H$, which completes the proof.
\end{proof} 

\begin{remark}
Let $K_G(d)$ denote the kernel of reduction modulo $d$ of $G(\hat{\Z})\to G(\Z/d\Z)$, for $G \in \{ \GL_2, \G \}$. If $U \subseteq K_G(d)$ for some $d \geq 3$, then the hypotheses of the lemma are satisfied with $U' = K_G(d)$.  
\end{remark}

We recall that both $\Sh_{\GL_2}$ and $\Sh_{\G}$ admit a description as moduli spaces of abelian schemes: given sufficiently small open compact subgroups $V \subseteq \GL_2(\Af)$ and  $U \subseteq \G(\Af)$, $\Sh_{\GL_2}(V)$ is the moduli of (isomorphism classes of) elliptic curves with $V$-level structure, while $\Sh_{\G}(U)$ parametrises (isomorphism classes of) principally polarised abelian schemes of relative dimension $3$ and $U$-level structure. 

Finally, we recall that, for $g \in \G(\Af)$ and $U$ sufficiently small, we have a map of schemes over $\Q$ \[ g: \Sh_{\G}(U) \to \Sh_{\G}(g^{-1} U g) \] given by $g \cdot [(z, h)] = [(z, h g)]$.  
For $g \in \G(\Af)$, we denote by $\iota^g_{U}$ the composition \[ \xymatrix{ \Sh_{\H}(gUg^{-1} \cap \H)  \ar[rr]^-{ \iota_{gUg^{-1}}} & &  \Sh_{\G}(gUg^{-1}) \ar[rr]^g && \Sh_{\G}(U) }.\] 

\begin{remark}\label{actiononmoduli}
For $U$ equal to the kernel of reduction modulo $d$, $U$-level structures of an abelian scheme $A$ correspond to bases of the $d$-torsion points of $A$. Note that the right-translation action of $g \in \GL_2(\widehat{\Z})$ (or $\G(\widehat{\Z})$) on the variety corresponds, at the level of moduli spaces, to the map $g \colon (A, \lambda, \{e_i\}) \to (A, \lambda, \{e'_i\})$, where $(e_i')=g^{-1} \cdot (e_i)$, where $\{e_i \}$ forms a basis of the $d$-torsion points for $A$.
\end{remark}

\subsection{Level structures}\label{levelstructures}

We introduce next several level structures that we will be using throughout. The reader is urged to skip this section and come back as the situation demands.

\begin{definition}\label{miraklingensub} Let $K^{(p)} \subset \G(\widehat{\Z}^{(p)})$ be a compact open subgroup satisfying the hypotheses of Lemma \ref{closedimm}. For any $n \in \N$, let  $K_n :=  K^{(p)} K_{n}^{p} \subseteq \G(\widehat{\Z})$, where
\[K_n^{p} := \{ g \in \G(\Z_p) : \quad R_6(g) \equiv (0, \ldots, 0, 1) \text{ mod } p^n\}, \]
and where $R_6(g)$ denotes the sixth row of $g$.

For any $n \in \N$, we let $K_1(n) = \mathrm{pr}_1(K_n \cap \H)$, where $\mathrm{pr}_1:\H \to \GL_2$ is the projection to the first $\GL_2$-component of $\H$. Observe that its component at $p$ is given by

\[ K_1^p(n):= \{ g \in \GL_2(\zp) \; \mid \; g \equiv I \text{ mod } \left[ \begin{smallmatrix} 1 &1 \\ {p^n} & {p^n} \end{smallmatrix} \right] \} . \]
We will always assume $K_1(n)$ is a sufficiently small compact open subgroup of $\GL_2(\widehat{\Z})$.
\end{definition}

\begin{remark} \leavevmode
\begin{itemize}
 \item Note that, at $p$, the level group $K_n \cap \H$ has component \[ K_1^p(n) \boxtimes \GL_2(\Z_p) \boxtimes \GL_2(\Z_p). \]
 \item If $K^{(p)} \times \G(\Z_p) = K_\G(d)$ for some integer $d \geq 3$ coprime to $p$, then $K_n$ and $K_1(n)= (\GL_2(\widehat{\Z}^{(p)}) \times K_1^p(n)) \cap K_{\GL_2}(d)$ are sufficiently small.
 \item By Lemma \ref{closedimm}, $\iota_{K_n}$ is a closed immersion and we get \[ \xymatrix{ \Sh_{\GL_2}(K_1(n)) & \Sh_{\H}(K_n \cap \H) \ar[l]_-{\mathrm{pr}_1} \ar[r]^-{\iota_{K_n}} & \Sh_{\G}(K_{n}),} \]
 where $\mathrm{pr}_1$ now denotes the morphism of Shimura varieties induced by the projection to the first $\GL_2$-component of $\H$.   This diagram will be fundamental in the definition of the motivic classes underlying our Euler system construction.
\end{itemize}
\end{remark}

Let $\eta$ be the co-character of the maximal torus of $\G$ defined by \[x \mapsto \left( \begin{smallmatrix} x^3  & {} & {} &  &  & \\ & x^2 &  &  & & \\  &   & x^2 &  & & \\  &   & & x & &{} \\  &   & & & x   &{} \\  &   &  & &  & 1 \end{smallmatrix} \right) \] and let $\eta_p:=\eta(p) \in {\G}(\Q_p) \subseteq \G(\Af)$.
 
 \begin{definition}
Recall that we denote  by $K_\G(p^m) \subseteq \G(\widehat{\Z})$ the kernel of the reduction modulo $p^m$. For $m \in \N$, define subgroups of ${\G}(\A_f)$

\begin{itemize}
 \item $K_{n,m(p)}':= K_n \cap \eta_p^{m+1} K_n \eta_p^{-(m+1)} \cap K_\G(p^m)$;
 \item $K_{n, m+1}':= K_{n,m(p)}' \cap K_\G(p^{m+1})$. 
 \end{itemize}
\end{definition}
 
\begin{remark} \leavevmode
 \begin{itemize}
 \item The group $K_{n,0(p)}'$ is the largest subgroup of $K_n$ such that right multiplication by $\eta_p$ induces a morphism
 \[ \eta_p: \Sh_{\G}( K_{n,0(p)}') \longrightarrow \Sh_{\G}(K_{n}).\]
 \item The definition of these last level groups will be justified by Lemma \ref{cartdiag}.
 \item  In other words, for $n > m$, these subgroups are defined as follows.
    \[ K_{n,m}' := \bigg\{ g \in K_0 \; \mid \; g \equiv I \text{ mod } \left[ \begin{smallmatrix} p^n & p^{m} &  p^{m} & p^{2m} & p^{2m} & p^{3m} \\ p^n & p^m & p^m & p^{m} & p^{m } & p^{2m}  \\ p^n & p^m & p^m & p^{m} & p^{m } & p^{2m } \\ p^n & p^m & p^m & p^m & p^m & p^{m} \\ p^n & p^m & p^m & p^m & p^m & p^{m}  \\ p^n & p^n & p^n & p^n & p^n & p^n  \end{smallmatrix}  \right] \bigg\}. \]
    \[ K_{n,m(p)}' := \bigg\{ g \in K_0 \; \mid \; g \equiv I \text{ mod } \left[ \begin{smallmatrix} p^n & p^{m+1} &  p^{m+1} & p^{2(m+1)} & p^{2(m+1)} & p^{3(m+1)} \\ p^n & p^m & p^m & p^{m+1} & p^{m + 1} & p^{2(m+1)}  \\ p^n & p^m & p^m & p^{m + 1} & p^{m + 1} & p^{2(m + 1)} \\ p^n & p^m & p^m & p^m & p^m & p^{m + 1} \\ p^n & p^m & p^m & p^m & p^m & p^{m + 1}  \\ p^n & p^n & p^n & p^n & p^n & p^n  \end{smallmatrix}  \right] \bigg\}. \]
\item Observe that we have a tower of inclusions \[K_n= K_{n,0}' \supseteq K_{n,0(p)}' \supseteq K_{n,1}' \supseteq K_{n,1(p)}' \supseteq K_{n,2}' \supseteq \ldots \]    
    \end{itemize}
\end{remark}

\subsection{Representations of algebraic groups}

We study now the branching laws for the restriction of an irreducible algebraic representation of $\G$ to some of its subgroups.

\subsubsection{Highest weight representations}

Recall that every irreducible algebraic representation of $\GL_2$ is of the form $\operatorname{Sym}^d \otimes \det^k$ for some  $d \in \N, k \in \Z,$ where $\operatorname{Sym}^d$ denotes the $d$-th symmetric power of the standard $\GL_2$-representation. We will next review the highest weight theory for the groups $\GSp_4$ and $\GSp_6$.

Let $T$ be the diagonal torus of ${\G}$ (which coincides with the diagonal torus of ${\H}$) and denote by $\chi_i \in X^\bullet(T)$, $1 \leq i \leq 6$, the characters of $T$ given by projection onto the $i$-th coordinate. We then have $\chi_i \chi_{7 - i} = \nu$, $i = 1, 2, 3$, where $\nu$ denotes the symplectic multiplier. We see $\GSp_4$ inside ${\G}$ and $\chi_i$, $i \in \{1 , 2 , 5, 6 \}$, denote as well the characters of its diagonal torus.

For $a,b$ non-negative integers, let $\mu = (\mu_1 \geq \mu_2)$, $\mu_2 = b, \mu_1 = a + b$ and denote by $V^\mu$ the unique (up to isomorphism) irreducible algebraic representation of $\GSp_4$ with highest weight $\chi_1^{\mu_1} \chi_2^{\mu_2}$ with central character $x \mapsto x^{|\mu|}$, where $|\mu| = \mu_1 + \mu_2$, which has dimension $\frac{1}{6} (a + 1)(b + 1)(a + b + 2)(a + 2b + 3)$. Similarly, given $a,b,c$ positive integers, let $\lambda = (\lambda_1 \geq \lambda_2 \geq \lambda_3)$, $\lambda_3 = c, \lambda_2 = b + c, \lambda_1 = a + b + c$ and denote by $V^\lambda$ the unique algebraic irreducible representation of ${\G}$ with highest weight  $\chi_1^{\lambda_1} \chi_2^{\lambda_2} \chi_3^{\lambda_3}$ and central character $x \mapsto x^{|\lambda|}$, where $|\lambda| = \lambda_1 + \lambda_2 + \lambda_3$, which is of dimension $\frac{1}{720}(a + 1) (a + 2(b + c) + 5) (a + b + 2) (a + b + 2c + 4) (b + 1) (b + 2c + 3)(a+b+c+3)(b+c+2)(c+1)$. 

\subsubsection{Branching laws} \label{sect:branching}

For $\lambda = (\lambda_1 \geq \lambda_2 \geq \lambda_3)$ and $\mu = (\mu_1 \geq \mu_2)$ as above, we say that $\mu$ doubly interlaces $\lambda$ if $\lambda_1 \geq \mu_1 \geq \lambda_3$ and $\lambda_2 \geq \mu_2 \geq 0$. We recall the following branching law result

\begin{proposition} \label{braching1} Let $\lambda = (\lambda_1 \geq \lambda_2 \geq \lambda_3 \geq 0)$ and $V^\lambda$ be as above. Then
\begin{itemize}
 \item We have a decomposition of $\Sp_4 \boxtimes \SL_2$-representations
 \[ V^{\lambda} = \bigoplus_{\mu} V^\mu \boxtimes (\operatorname{Sym}^{r_1} \otimes \operatorname{Sym}^{r_2} \otimes \operatorname{Sym}^{r_3}), \] where the sum is over all $\mu = (\mu_1 \geq \mu_2 \geq 0)$ doubly interlacing $\lambda$ and where $r_i = x_i - y_i$ for $\{ x_1 \geq y_1 \geq x_2 \geq y_2 \geq x_3 \geq y_3 \}$ being the decreasing rearrangement of $\{ \lambda_1, \lambda_2, \lambda_3, \mu_1, \mu_2, 0 \}$.
 \item We have a decomposition of $\SL_2 \boxtimes \SL_2$-representations

\[ V^\mu = \bigoplus_{x = 0}^{\mu_1 - \mu_2} \bigoplus_{y = 0}^{\mu_2} \operatorname{Sym}^{\mu_1-x - y} \boxtimes \operatorname{Sym}^{\mu_2 - y + x}. \]
\end{itemize}
\end{proposition}

\begin{proof}
 The first statement is just \cite[Theorem 3.3]{WallachYacobi}. We sketch a proof of the second point, which is stated in \cite[Proposition 4.3.1]{LSZ1}. For the parametrization of the special case of $\GSp_4$, applying \cite[Theorem 3.3]{WallachYacobi} we obtain
 \begin{align*} V^\mu &= \oplus_{s = 0}^{\mu_1} (\operatorname{Sym}^{r_1} \otimes \operatorname{Sym}^{r_2}) \boxtimes \operatorname{Sym}^s \\
  & = \oplus_{s = 0}^{\mu_1} (\operatorname{Sym}^{r_1 + r_2} \oplus \operatorname{Sym}^{r_1 + r_2 - 2} \oplus \hdots \oplus \operatorname{Sym}^{|r_1 - r_2|}) \boxtimes \operatorname{Sym}^s.
 \end{align*}
 Observe that every factor appears with multiplicity one. Dividing the sum for $0 \leq s \leq \mu_2$ and $\mu_2 < s \leq \mu_1$ we see that $r_1 = \mu_1 - \mu_2$, $r_2 = s$ and $r_1 = \mu_1 - s, r_2 = \mu_2$ respectively. Drawing the points $(x, y)$ such that the representation $\operatorname{Sym}^x \boxtimes \operatorname{Sym}^y$ appears in the above sum we see that we get every integer pair $(x, y) \in \Z^2$ with $x + y \equiv \mu_1 + \mu_2 \text{ (mod } 2)$ inside the rectangle with vertices $(0, \mu_1 - \mu_2)$, $(\mu_1 - \mu_2, 0)$, $(\mu_2, \mu_1)$ and $(\mu_1, \mu_2)$. Choosing the right parametrisation of these points (i.e. taking $(\mu_2, \mu_1)$ as the origin) we get the desired expression. 
\end{proof}

For a fixed $k$ and $\lambda$, we are interested in studying how many $\H$-representations of the form $ \operatorname{Sym}^{(k, 0, 0)} := \operatorname{Sym}^k \boxtimes \operatorname{Sym}^0 \boxtimes \operatorname{Sym}^0 $ appear in the decomposition of the restriction of $V^\lambda$ to $\H$. It will be useful to consider the obvious factorisation of our embedding $\H \subseteq \G$ through $\H' := \GSp_4 \boxtimes \GL_2$; this is because any irreducible $\H'$-factor of an irreducible $\G$-representation will have multiplicity one.

\begin{lemma} \label{branching1b}
The sum of all irreducible sub-$\H'$-representations of $V^\lambda$ isomorphic (up to a twist) to $V^\mu \boxtimes \operatorname{Sym}^0$ for some $\mu$ is given by
\[ \bigoplus_{\mu \in \mathcal{A}(\lambda)} (V^\mu \boxtimes \operatorname{Sym}^0) \otimes \nu^{\frac{|\lambda| - |\mu|}{2}}, \]
where $\mathcal{A}(\lambda) \subseteq \Z^2$ denotes the region of points $(\mu_1, \mu_2) \in \Z^2$ satisfying $|\mu| \equiv |\lambda| \, (\text{mod } 2)$ and lying in the rectangle defined by the inequalities
\[ 
\begin{cases}
  \mu_1 - \mu_2 \leq \lambda_1 - \lambda_2 + \lambda_3, \\
  \mu_1-\mu_2 \geq | \lambda_1 - \lambda_2 - \lambda_3 |, \\
  \mu_1 + \mu_2 \geq \lambda_1 - \lambda_2 + \lambda_3, \\
  \mu_1 + \mu_2 \leq \lambda_1 + \lambda_2 - \lambda_3.
\end{cases}
\]
\end{lemma}

\begin{proof}
 Applying Proposition \ref{braching1}, we obtain a decomposition as $\Sp_4 \boxtimes \SL_2$-representations
 \begin{eqnarray*}
 V^\lambda &=& \bigoplus_\mu V^\mu \boxtimes (\operatorname{Sym}^{r_1} \otimes \operatorname{Sym}^{r_2} \otimes \operatorname{Sym}^{r_3} ) \\
 &=& \bigoplus_\mu \bigoplus_{i = 0}^{\min(r_1, r_2)} V^\mu \boxtimes (\operatorname{Sym}^{r_1 + r_2 - 2i} \otimes \operatorname{Sym}^{r_3} ) \\
 &=& \bigoplus_\mu \bigoplus_{i = 0}^{\min(r_1, r_2)} \bigoplus_{j = 0}^{\min(r_3, r_1 + r_2 - 2i)} V^\mu \boxtimes (\operatorname{Sym}^{r_1 + r_2 + r_3- 2i - 2j} ), \\
 \end{eqnarray*}
 where the sum is over all $\mu = (\mu_1 \geq \mu_2 \geq 0)$ doubly interlacing $\lambda$ and where $r_i = x_i - y_i$ for $\{ x_1 \geq y_1 \geq x_2 \geq y_2 \geq x_3 \geq y_3 \}$ being the decreasing rearrangement of $\{ \lambda_1, \lambda_2, \lambda_3, \mu_1, \mu_2, 0 \}$.
 
 We deduce that if $V^\mu \boxtimes \operatorname{Sym}^0$ appears as a sub-$\Sp_4 \boxtimes \SL_2$-representation then 
 \[ r_1 + r_2 - 2i - j = 0, \;\;\; r_3 - j = 0, \]
 which implies $j = r_3$ and $2 i = r_1 + r_2 - r_3$ and hence, since $0 \leq i \leq \min(r_1, r_2)$,
 \begin{equation} \label{ineqweights}
r_1 + r_2 \geq r_3 \geq r_1 + r_2 - 2 \min(r_1, r_2) = |r_1 - r_2|
 \end{equation}
 and $r_1 + r_2 + r_3 \equiv 0 \, (\text{mod } 2)$, which is equivalent to saying that $|\mu| \equiv |\lambda| \, (\text{mod } 2)$. 
The result follows by unfolding the two inequalities of \eqref{ineqweights}. 
\end{proof}

\begin{lemma} \label{branching2b}
The sum of all irreducible sub-$\H$-representations of $V^\lambda$ isomorphic (up to a twist) to $\operatorname{Sym}^{(k, 0, 0)}$ for some $k \geq 0$ is given by

\[ \bigoplus_{\substack{k = | \lambda_1 - \lambda_2 - \lambda_3 | \\ k \equiv |\lambda| \, (\text{mod } 2)}}^{\lambda_1 - \lambda_2 + \lambda_3} r \cdot \operatorname{Sym}^{(k, 0, 0)} \otimes \det^{\frac{|\lambda| - k}{2}}, \]
for  $r= \lambda_2 - \lambda_3 +1$.
\end{lemma}

\begin{proof}
 This follows immediately from Lemma \ref{branching1b}. Indeed observe that, by Proposition \ref{braching1}, for any $\mu$, the unique sub-$\SL_2 \boxtimes \SL_2$-representation of $V^\mu$ of the form $\operatorname{Sym}^{(k,0)}$ is $\operatorname{Sym}^{(\mu_1 - \mu_2, 0)}$. The result then follows by analysing the possible values of $\mu_1 - \mu_2$ in the region $\mathcal{A}(\lambda)$ of the above lemma. The value $r$ is the number of $(\mu_1, \mu_2) \in \mathcal{A}(\lambda)$ such that $\mu_1 - \mu_2 = k$, i.e. the length of one of the sides of the rectangle, forming the boundary of $\mathcal{A}(\lambda)$. The twist is there so that the central characters of $\operatorname{Sym}^{(k, 0, 0)}$ and $V^\lambda$ and the inclusion is $\H$-equivariant.
\end{proof}

\begin{remark}
The values of $k$ and $r$ can be easily deduced by drawing the region $\mathcal{A}(\lambda)$. 
 For instance, from  Figure \ref{areawithnotrivialcoeff} for $\lambda=(9,6,2)$, we have that $\operatorname{Sym}^{(k, 0, 0)}$ appears in the decomposition of the restriction of $V^\lambda$ to $\H$ only if $k \in \{ 1 , 3 ,5 \}$ with multiplicity $r=5$. 
 \end{remark}

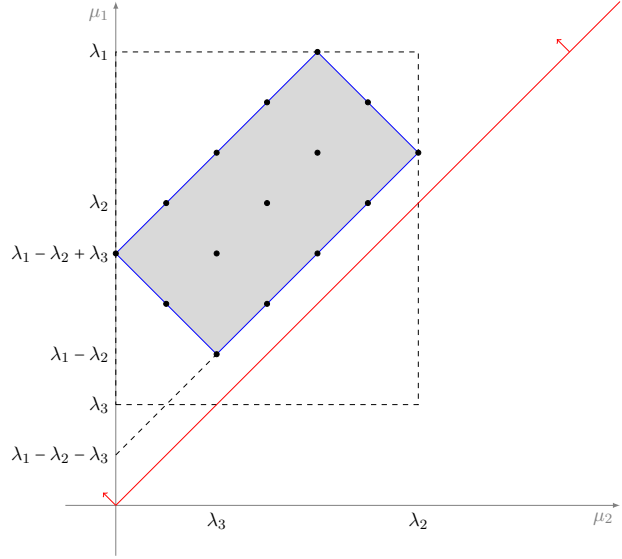
\begin{figure}[ht] 
 \resizebox{0.5\linewidth}{!}{

    \begin{tikzpicture}
    \coordinate (Origin)   at (0,0);
    \coordinate (XAxisMin) at (-1,0);
    \coordinate (XAxisMax) at (10,0);
    \coordinate (YAxisMin) at (0,-1);
    \coordinate (YAxisMax) at (0, 10);
    \coordinate (ylambda3) at (0,2);
    \coordinate (xlambda3) at (2,0);
    \coordinate (ylambda2) at (0,6);
    \coordinate (xlambda2) at (6,0);
     \coordinate (ylambda1) at (0,9);
    \coordinate (xlambda1) at (9,0);
     \coordinate (ylambda1-2+3) at (0,5);
    \coordinate (lambda2lambda1) at (6,9);
    \coordinate (lambda2lambda1+3) at (6,11);
    \coordinate (lambda2lambda1-3) at (6,7);
    \coordinate (ylambda1-2-3) at (0,1);
    \draw [thin, gray,-latex] (XAxisMin) -- (XAxisMax) node [below left] {$\mu_2$};
    \draw [thin, gray,-latex] (YAxisMin) -- (YAxisMax) node [below left] {$\mu_1$} ;
    \clip (-3,-1) rectangle (10cm,10cm); 
    \draw[dashed] (ylambda3) node [left] {$\lambda_3$} rectangle (lambda2lambda1);
   \draw (ylambda1) node [left] {$\lambda_1$};
   \draw (ylambda2) node [left] {$\lambda_2$};
   \draw (ylambda1-2+3) node [left] {$\lambda_1 - \lambda_2 +\lambda_3$};
   \draw (ylambda1-2-3) node [left] {$\lambda_1 - \lambda_2 -\lambda_3$};
\draw (xlambda2) node [below] {$\lambda_2$};
\draw (xlambda3) node [below] {$\lambda_3$};
\draw (0,3) node [left] {$ \lambda_1 - \lambda_2$};
\draw[dashed] (ylambda1-2-3) -- (2,3);
\filldraw[fill=gray, fill opacity=0.3, draw=blue] (ylambda1-2+3) -- (4,9)  -- (6,7) -- (lambda2lambda1-3) -- (2,3) --  (ylambda1-2+3);
\draw[thin, red] (0,0) -- (10,10) ;
\draw[-> , red] (0,0) -- (-1/4,1/4); 
\draw[-> , red] (9,9) -- (9-1/4,9+1/4);

   \foreach \t in {1,3}{
      \foreach \x in {2,3,...,5}{        
             \node[draw,circle,inner sep=1pt,fill] at (\x,\t+\x) {};
                 }
             }
             
      \foreach \x in {0,1,...,4}{        
       \node[draw,circle,inner sep=1pt,fill] at (\x,5+\x) {};
       }
 \node[draw,circle,inner sep=1pt,fill] at (1,4) {};
  \node[draw,circle,inner sep=1pt,fill] at (6,7) {};
  \end{tikzpicture}
  }
  \caption{The region $\mathcal{A}(\lambda)$ for $\lambda=(9,6,2)$.}
  \label{areawithnotrivialcoeff}
\end{figure}

\subsubsection{Integral structures} \label{branchingint2}

Denote by $\mathfrak{h}, \mathfrak{g}$ the Lie algebras of $\H$ and $\G$ respectively, and write $U(\mathfrak{h}), U(\mathfrak{g})$ for their universal enveloping algebras. For $\mathfrak{a} \in \{ \mathfrak{h}, \mathfrak{g} \}$, denote by $U_\Z(\mathfrak{a})$ the Kostant $\Z$-form in $U(\mathfrak{a})$ (\cite[Chapter 2]{Steinberg}), which is some subring of $U(\mathfrak{a})$ generated over $\Z$ by an explicit family of ordered monomials given in terms of the choice of a Chevalley basis of $\mathfrak{a}$ (which also forms a $PBW$-basis of $U(\mathfrak{a})$), so that $U(\mathfrak{a})$ is obtained from $U_\Z(\mathfrak{a})$ by base-change.

For an $\mathfrak{a}$-module $V$, an admissible lattice $V_\Z$ in $V$ is a $\Z$-lattice which is stable under the action of $U_\Z(\mathfrak{a})$. By \cite[Corollary 1]{Steinberg}, we know that admissible lattices exist for any representation of a semi-simple Lie group, and that such a lattice is the direct sum of its weight components. For a weight $\lambda$, fix a highest weight vector $v^\lambda$ of weight $\lambda$ and consider $V_\Z^\lambda$ the maximal admissible lattice inside $V^\lambda$ whose intersection with the highest weight space is $\Z \cdot v^\lambda$. Observe that $V_\Z^\lambda$ is also an admissible lattice considered as an ${\H}$-representation (since $U_\Z(\mathfrak{h}) \subseteq U_\Z(\mathfrak{g})$, which can be seen using \cite[Theorem 2]{Steinberg} and the fact that a set of simple roots for $\mathfrak{h}$ can be extended to a set of simple roots of $\mathfrak{g}$ and that their Cartan subalgebras coincide).

Let $\langle e_1,e_2,e_3,f_3,f_2,f_1 \rangle$ be a symplectic basis for the standard $\G$-representation $V^{(1 \geq 0 \geq 0)}$. \\ Denote $\operatorname{Sym}^{(k, 0, 0)}_\Z \subseteq (\operatorname{Sym}^{(k, 0, 0)} \cap V^\lambda_\Z)$ the minimal admissible lattice of $\operatorname{Sym}^{(k, 0, 0)}$ such that the intersection $\operatorname{Sym}^{(k, 0, 0)}$ with its highest weight space is $\Z \cdot e_1^k$ (it is isomorphic to the algebra of symmetric tensors $\operatorname{TSym}_\Z^k \boxtimes \operatorname{TSym}^0_\Z \boxtimes \operatorname{TSym}_\Z^0$).

By \cite[Corollary 1]{Steinberg} (cf. also {\cite[Corollary 1 to Theorem 1]{Kostant}), the restriction to $\H$ of the lattice $V_\Z^\lambda$ decomposes as the direct sum of its highest weight components.
In particular, for every $\mu = (\mu_1, \mu_2) \in \mathcal{A}(\lambda)$ and $k = \mu_1 - \mu_2$, we have that $(\operatorname{Sym}^{(k, 0, 0)} \otimes \operatorname{det}^{\frac{|\lambda| - k}{2}} ) \cap V^\lambda_\Z \subseteq V^{\lambda}_\Z$ is non empty. By fixing any highest weight vector $v^{[\lambda, \mu]}$ in this sub-lattice, we can define a homomorphism of $\H$-representations
\[ \br^{[\lambda,\mu]}_\Z: \operatorname{Sym}^{(k, 0, 0)}_\Z \otimes \operatorname{det}^{\frac{|\lambda| - k}{2}} \to V^{\lambda}_\Z, \] by sending $e_1^k$ to $v^{[\lambda, \mu]}$.

\subsection{Gysin morphisms}

In the next section, we will define \'etale and motivic classes in the cohomology of the $\G$-Shimura variety with coefficients by taking the image under Gysin morphisms of certain classes in the cohomology of the $\H$-Shimura variety. To define these maps, we will translate the branching laws for algebraic representations of $\H$ and $\G$ described above into a statement for the corresponding \'etale sheaves and relative Chow motives on the Shimura varieties.
Let us briefly recall the main properties of the functor defined in \cite{Ancona}. For a reductive group $G$ over $\Q$, denote by $\operatorname{Rep}_{\Q}(G)$ the category of representations of $G$ over $\Q$.
Moreover, for a smooth quasi-projective scheme $S$ over a field of characteristic zero, let $\operatorname{CHM}_{\Q}(S)$ denote the $\Q$-linear tensor pseudo-abelian category of relative Chow motives over $S$. Recall that there is a functor $M$ from the category of smooth projective schemes over $S$ to $\operatorname{CHM}_{\Q}(S)$; let $\mathbbm{1}_S:=M(S)$, and denote by $\mathbb{L}_S$ the Lefschetz motive appearing in the decomposition of $M(\mathbb{P}^1_S)$ as  $\mathbbm{1}_S \oplus \mathbb{L}_S$. For any positive integer $m$ and $\mathscr{V} \in \operatorname{Ob}(\operatorname{CHM}_{\Q}(S))$, we denote by $\mathscr{V}(-m)$ and  $\mathscr{V}(m)$ the tensor products of $\mathscr{V}$ with $\mathbb{L}_S^{\otimes m}$ and $(\mathbb{L}_S^\vee )^{\otimes m}$.
In order to define Ancona's functor, recall the following. 
\begin{proposition}[\cite{DeningerMurre}] Let $\pi:A\to S$ be an abelian scheme of relative dimension $g$; there exists a decomposition in $\operatorname{CHM}_{\Q}(S)$ \[ M(A)= \bigoplus_{i=0}^{2g} h^i(A),\] where $[n]^*$ acts on $h^i(A)$ as multiplication by $n^i$ and the $\ell$-adic realisation of $h^i(A)$ is $R^i \pi_* \Q_{\ell}$. 
\end{proposition}

Now, consider a Shimura datum $(G,X)$ of PEL-type. For any sufficiently small level subgroup $U \subseteq G(\Af)$ there is a Shimura variety $\Sh_G(U)$, which admits a model over the reflex field of $(G,X)$, and a universal abelian scheme $\mathscr{A}/\Sh_G(U)$ with PEL structure. 

\begin{proposition}[\cite{Ancona}] \label{Ancona}  There is a tensor functor \[ \mu^G_U: \operatorname{Rep}_{\Q}(G) \longrightarrow \operatorname{CHM}_{\Q}(\Sh_G(U)),\] which respect duals and satisfies the following:

\begin{enumerate}
\item If $V$ is the standard representation of $G$, then $\mu^G_U(V)=h^1(\mathscr{A})$;
\item If $\nu: G \to \Gm$ is the multiplier, then $\mu^G_U(\nu)=\mathbb{L}_{\Sh_G(U)}$;
\item for any prime $p$, the $p$-adic \'etale realisation of $\mu^G_U(V)$ is the \'etale sheaf associated to $V \otimes \Q_p$ (cf. \cite{Pink}), with $U$ acting on the left via $U \hookrightarrow G(\Af) \to G(\Q_p)$.
\end{enumerate}
\end{proposition}

\begin{remark}\label{tatemoduledualrep}
We have adopted conventions used in \cite{LSZ1}. This is coherent with the fact that, in the case of $\GL_2$, by Remark \ref{actiononmoduli}, the $p$-adic Tate module $T_p\mathscr{E}$ of the universal elliptic curve $\mathscr{E}$ corresponds to the dual of the standard representation of $\GL_2(\Z_p)$.  Thus, $T_p\mathscr{E}$ gives a lattice in the $p$-adic \'etale realisation of $h^1(\mathscr{E})^{\vee}$.
\end{remark}

As explained in \cite[\S 6.2]{LSZ1}, there is a canonical $G(\Af)$-equivariant structure on $\mu^G_U(V)$ for every $V$ in $\operatorname{Rep}_{\Q}(G)$, which is compatible with the $G(\Af)$-equivariant structure on  the corresponding $p$-adic \'etale realisations. Thus, we have a functor \[\mu^G: \operatorname{Rep}_{\Q}(G) \longrightarrow \operatorname{CHM}_{\Q}(\Sh_G)^{G(\Af)}, \] where $\Sh_G= \varprojlim_U \Sh_G(U)$.

\begin{proposition} \label{branchingmotivic}
There is a commutative diagram of functors \[\xymatrix{\operatorname{Rep}_{\Q}(\G) \ar[rr]^-{\mu^{\G}} \ar[d]_{\bullet_{ |_{\H}}} & & \operatorname{CHM}_{\Q}(\Sh_{\G})^{\G(\Af)} \ar[d]^{\Delta^*} \\ \operatorname{Rep}_{\Q}(\H) \ar[rr]^-{\mu^{\H}} & & \operatorname{CHM}_{\Q}(\Sh_{\H})^{\H(\Af)},} \] 
where $\Delta^*$ denotes pull-back.
\end{proposition}
\begin{proof} 
This is stated in \cite[Proposition 6.2.5]{LSZ1} and a proof will appear in forthcoming work of Alex Torzewski.
\end{proof}
 
Let $U \subseteq \G(\Af)$ be a sufficiently small open compact subgroup; as in \cite[\S 4.1]{lemmarf}, one has motivic cohomology groups $H^{\bullet}_{\operatorname{mot}}(\Sh_{\G}(U), \mathscr{V}_{\Q}(m)),$ for any  $\mathscr{V}_{\Q}\in \operatorname{Ob}(\operatorname{CHM}_{\Q}(\Sh_{\G}(U))$ and any integer $m$. Suppose $U \subseteq \G(\Af)$ is chosen so that $\iota_U$ is a closed immersion (e.g. Lemma \ref{closedimm}), then we have Gysin morphisms \[\iota_{U,*}: H^i_{\operatorname{mot}}(\Sh_{\H}(U \cap \H), \Delta^*\mathscr{V}_{\Q}(m)) \longrightarrow H^{i+6}_{\operatorname{mot}}(\Sh_{\G}(U), \mathscr{V}_{\Q}(3+m)).\]

We want to compose these maps $\iota_{U,*}$ with the maps in cohomology coming from the branching laws in $\operatorname{Rep}_{\Q}(\H)$ described above.

\begin{definition} \leavevmode
\begin{itemize} 
\item Let $\mathscr{W}^{\lambda}_{\Q}$ be the relative Chow motive $\mu^{\G}(W^\lambda)$ over $\Sh_{\G}$, where $W^\lambda$ is the algebraic representation of $\G$ given by $V^{\lambda} \otimes \nu^{-|\lambda|}$.
\item Let $\mathscr{H}^{(k,0,0)}_{\Q}$ be the relative Chow motive  $\mu^{\H}(\operatorname{Sym}^{(k,0,0)} \otimes \operatorname{det}^{-k})$ over $\Sh_{\H}$.
\end{itemize}
\end{definition}

\begin{proposition} Let $\mu=(\mu_1 \geq \mu_2) \in \mathcal{A}(\lambda)$ and let $k=\mu_1-\mu_2$; we have
\[ \iota_{U, *}^{[\lambda,\mu]} : H^{\bullet}_{\operatorname{mot}}(\Sh_{\H}(U \cap \H), \mathscr{H}^{(k,0,0)}_{\Q}(\star)) \longrightarrow H^{\bullet+6}_{\operatorname{mot}}(\Sh_{\G}(U), \mathscr{W}^{\lambda}_{\Q}(\star+3+\tfrac{k-|\lambda|}{2})).\] 
\end{proposition}

\begin{proof}
Note that by \S \ref{branching2b}, we have \[ \operatorname{Sym}^{(k,0,0)} \otimes \operatorname{det}^{\tfrac{|\lambda|-k}{2}} \hookrightarrow V^{\lambda}.\]  After twisting it, this gives a map \[ \operatorname{Sym}^{(k,0,0)} \otimes \operatorname{det}^{-k} \hookrightarrow V^{\lambda} \otimes \nu^{-|\lambda|+\tfrac{|\lambda|-k}{2}}=W^{\lambda} \otimes \nu^{\tfrac{|\lambda|-k}{2}}.\]
By Proposition \ref{branchingmotivic} and Proposition \ref{Ancona}(2), we get a morphism
\[ \br^{[\lambda,\mu]}: \mathscr{H}^{(k,0,0)}_{\Q} \longrightarrow \Delta^*\mathscr{W}^{\lambda}_{\Q}(-\tfrac{|\lambda|-k}{2}). \]
The composition of the corresponding map in cohomology $\br^{[\lambda,\mu]}$ with $\iota_{U,*}$ defines the desired map $\iota_{U, *}^{[\lambda,\mu]}$.
\end{proof}

\begin{remark}\leavevmode 
\begin{itemize}
\item We have \'etale regulator maps \[ r_{\et}:H^{j}_{\operatorname{mot}}(\Sh_{\G}(U), \mathscr{W}^{\lambda}_{\Q}(\star)) \longrightarrow H^{j}_{\et}(\Sh_{\G}(U), \mathscr{W}^{\lambda}_{\Q_p}(\star)), \] where  $ \mathscr{W}^{\lambda}_{\Q_p}$ is the the $p$-adic \'etale sheaf associated to $W^\lambda_\qp$.
\item From \S \ref{branchingint2}, we have "integral" Gysin morphisms in \'etale cohomology. Let $\mathcal{H}^{(k,0,0)}_{\Z_p}$  (resp. $\mathcal{W}^{\lambda}_{\Z_p}$) denote the $\Z_p$-sheaf associated to the lattice $\operatorname{Sym}^{(k,0,0)}_\Z \otimes \operatorname{det}^{-k}$ (resp. $V^\lambda_{\Z} \otimes \nu^{-|\lambda|} $), then we have \[ \iota_{U, *}^{[\lambda,\mu]} :H^{\bullet}_{\et}(\Sh_{\H}(U \cap \H), \mathcal{H}^{(k,0,0)}_{\zp}(\star)) \longrightarrow H^{\bullet+6}_{\et}(\Sh_{\G}(U), \mathcal{W}^{\lambda}_{\zp}(\star+3+\tfrac{k-|\lambda|}{2})).\] 
\end{itemize}
\end{remark}

\section{Definition of the classes} \label{defclasses}

This is the main section of our text. We give the definition of the zeta classes and we study their norm compatibility as we vary the level of the Shimura variety.

\subsection{Siegel units}

Recall that when $U$ is a sufficiently small open compact subgroup of $\GL_2(\A_f)$, the modular curve $\Sh_{\GL_2}(U)$ is a fine moduli space with universal elliptic curve $\EE \to \Sh_{\GL_2}(U)$. In particular $\Sh_{\GL_2}(K_1(n))$ is the moduli of isomomorphism classes of $(E,P_n,\alpha)$, where $P_n$ is an $p^n$-torsion point of the elliptic curve $E$ and $\alpha$ is a level $ \mathrm{pr}_1( K^{(p)} \cap \H)$-structure on $E$. Denote by $(\mathscr{E},e_n,\alpha) / \Sh_{\GL_2}(K_1(n))$ the universal object of $\Sh_{\GL_2}(K_1(n))$.
For an auxiliary positive integer $c$ coprime to $6$, let $_c \theta_{\mathscr{E}} \in \O(\mathscr{E}  \smallsetminus {\mathscr{E}} [c])^*$ be the norm compatible unit of \cite[Proposition 1.3(1)]{Kato}.

\begin{definition}\label{defeisclassesplusfootnote}
Define $_cg_n  :=(e_n)^*( {_c\theta_{\mathscr{E} }})\in \O(\Sh_{\GL_2}(K_1(n)))^*.$ \footnote{ If $n=0$, the above definition is not what one wants, since we would be pulling-back by the zero section. One can easily deal with this by fixing an auxiliary non-zero torsion section $x \in \mathscr{E}(\Sh_{\GL_2}(K_1(0)))$ of order an integer $ N \geq 1$ coprime to $p$, and pull-back  ${_c\theta_{\mathscr{E} }}$ by $x+e_n$.}
\end{definition}

\subsection{Higher weight Eisenstein classes}

Let $\mathcal{H}^k_{\Q}$ denote the relative Chow motive over $\Sh_{\GL_2}(K_1(n))$ associated to the $\GL_2$-representation $\operatorname{Sym}^k \otimes \operatorname{det}^{-k}$, for $k\geq 0$. For $k \geq 0$, Beilinson constructed motivic Eisenstein classes
\[ \operatorname{Eis}^k_n \in H^1_{\operatorname{mot}}(\Sh_{\GL_2}(K_1(n)), \mathcal{H}^k_{\Q}(1)). \]

\begin{remark} For $k=0$, $H^1_{\operatorname{mot}}(\Sh_{\GL_2}(K_1(n)),\Q(1)) = \O(\Sh_{\GL_2}(K_1(n)))^* \otimes \Q$ and  $\operatorname{Eis}^0_{n}$ is ${}_cg_{n}\otimes \frac{1}{c^2-1}$, for $c \ne 1$ an integer coprime to $6$ and congruent to 1 modulo $p^n$. 
\end{remark}

We denote by $\operatorname{Eis}^k_{\et,n}$ the image of the (motivic) Eisenstein class under the \'etale regulator. Kings has constructed an underlying integral \'etale Eisenstein class. Let $ \mathcal{H}^k_{\Z_p}$ be the $\zp$-sheaf associated to the minimal admissible lattice $\operatorname{Sym}^{k}_\Z \otimes \operatorname{det}^{-k}$ of $\operatorname{Sym}^k \otimes \operatorname{det}^{-k}$.

\begin{proposition}[\cite{kings}]\label{integralEisclasses}
There exists, for any $c$ be coprime with $6p$ and $k \geq 0$, an element
\[\operatorname{_cEis}^k_{\et,n} \in H^1_{\et}(\Sh_{\GL_2}(K_1(n)), \mathcal{H}^k_{\Z_p}(1))\]
such that  ${}\operatorname{_cEis}^k_{\et,n} = \left( c^2-c^{-k}\left( \begin{smallmatrix} c & 0 \\ 0 & c \end{smallmatrix} \right)^{-1} \right) \operatorname{Eis}^k_{\et,n}$
as elements of $H^1_{\et}(\Sh_{\GL_2}(K_1(n)), \mathcal{H}^k_{\Q_p}(1))$. For $k=0$, $\operatorname{_cEis}^0_{\et, n} = \partial({}_cg_n)$, where $\partial$ denotes the Kummer map.
\end{proposition}

\subsection{The classes at level $K'_{n,m}$}

We first construct classes in the cohomology of the Shimura variety of level $K_{n, m}'$.

Lemma \ref{cartdiag} below is the key ingredient for proving the \emph{vertical} norm relations of our classes (Theorem \ref{vernormrel}) and, indeed, it constitutes the main motivation for working with the level $K_{n,m}'$.

\begin{lemma}\label{cartdiag} Let $n,m \geq 1$ be such that $n\geq 3m+3$. There exists an element $u \in {\G}(\Af)$ such that the commutative diagram 
\begin{eqnarray}\label{diagramwithnm} \xymatrix{  & & & \Sh_{\G}(K_{n,m+1}') \ar[d]^{pr}  \\ 
\Sh_{\H}(uK_{n,m+1}'u^{-1} \cap {\H})  \ar[urrr]^-{\iota^u_{K_{n,m+1}'}}  \ar[rrr]^-{pr \circ \iota^u_{K_{n,m+1}'}}  \ar[d]_{\pi_p} & & &  \Sh_{\G}(K_{n,m(p)}') \ar[d]^{\pi_p'} \\
\Sh_{\H}(uK_{n,m}'u^{-1} \cap {\H})  \ar[rrr]^-{ \iota^u_{K_{n,m}'}} & & &  \Sh_{\G}(K_{n,m}')}  \end{eqnarray}
has Cartesian bottom square, where $\pi_p,$ $\pi_p'$ and $pr$ denote the natural projections.
\end{lemma}

\begin{remark} \leavevmode
\begin{enumerate}
\item The choice of $u$ does not depend on either $m$ or $m$ and it is not unique. We can take $u \in {\G}(\widehat{\Z})$, whose component at $p$ equals to $\left( \begin{smallmatrix} I & T \\ 0&I \end{smallmatrix} \right),$ with \[T = \left( \begin{smallmatrix} 1&1&0 \\ 1&0&1 \\ 0&1&1 \end{smallmatrix} \right), \]
and having trivial components elsewhere.
\item A proof of Lemma \ref{cartdiag} is a direct and not very pleasant calculation and it is given in Appendix A.
\end{enumerate}
\end{remark}

We now define the push-forward classes in the cohomology of the $\G$-Shimura variety of level $K_{n,m}'$.
Notice that we have a projection $\operatorname{pr}_{1,n,m}: \operatorname{Sh}_{\H}(uK_{n,m}'u^{-1}\cap \H) \longrightarrow \operatorname{Sh}_{\GL_2}(K_1(n))$. Moreover, for any $g \in \G(\Af)$ denote by $\iota_{K_{n,m}', g, *}^{[\lambda, \mu]}$ the composition $g_* \circ \iota_{gK_{n,m}'g^{-1}, *}^{[\lambda, \mu]}$.

\begin{definition}\label{newdefofclasses} Let $V^\lambda$ be the irreducible representation of $\G$ of highest weight $\lambda = (\lambda_1 \geq \lambda_2 \geq \lambda_3)$, $\mu = (k + j \geq j) \in \mathcal{A}(\lambda)$ and let $n, m \in \N$.
\begin{itemize}
\item Let $\tilde{\mathscr{Z}}^{[\lambda, \mu]}_{n,m}$ be the class given by
\[ \iota_{K_{n,m}', u, *}^{[\lambda, \mu]} \circ \operatorname{pr}_{1,n,m}^* (\operatorname{Eis}^{k}_{n}) \in H^{7}_{\operatorname{mot}}(\Sh_{\G}(K_{n,m}'), \mathscr{W}^{\lambda}_{\Q}(4+\tfrac{k - |\lambda|}{2})). \]
\item Let $\tilde{z}^{[\lambda,\mu]}_{n,m}$ be the class  
\[ r_{\et}(\tilde{\mathscr{Z}}^{[\lambda, \mu]}_{n,m}) \in H^{7}_{\et}(\Sh_{\G}(K_{n,m}'), \mathcal{W}^{\lambda}_{\qp}(4+\tfrac{k - |\lambda|}{2})).\] 
\end{itemize}
\end{definition}
The motivic classes defined above are not a priori integral, which is due to a lack of theory of integral motivic Eisenstein classes. Building on Proposition \ref{integralEisclasses}, we give an integral construction of the $p$-adic \'etale classes as follows. This is better suited for studying norm relations in the $p$-cyclotomic tower and $p$-adic interpolation properties.

\begin{definition}\label{integraldefofclasses}
Let $_c \tilde{z}^{[\lambda,\mu] }_{n,m}$ be the class given by
\[ \iota_{K_{n,m}', u, *}^{[\lambda,\mu]} \circ \operatorname{pr}_{1,n,m}^* ({}\operatorname{_cEis}^{k}_{\et , n}) \in H^{7}_{\et}(\Sh_{\G}(K_{n,m}'), \mathcal{W}^{\lambda}_{\Z_p}(4+\tfrac{k-|\lambda|}{2})).\] 
\end{definition}

\subsection{The level groups $K_{n, m}$}\label{levelgroups}

Let $n \in \N$ and denote by $K_{n, 0} \subseteq {\G}(\widehat{\Z})$ the subgroup of $K_n$ defined by
\[K_{n, 0} := K_{n} \cap \bigg\{ g \in {\G}(\zp) \; \mid \; g \equiv I \text{ mod } \left[ \begin{smallmatrix} 1 & p &  p & p & p & p \\ p & 1 &  p & p & p & p  \\ p & p &  1 & p & p & p \\ 1 & 1 &  1 & p & p & p \\ 1 & 1 &  1 & p & p & p  \\ 1 & 1 & 1 & p & p & p   \end{smallmatrix}  \right] \bigg\}.\]

\begin{remark} The definition of $K_n$ is motivated by the proof of Theorem \ref{vernormrel}.
\end{remark}

For $m, n \in \N$, we aim to define classes in the cohomology of \[ \Sh_{\G}(K_{n, 0}) \times_{\operatorname{Spec}(\Q)} \operatorname{Spec}(\Q(\zeta_{p^m})). \]

\begin{definition} 
Let $n, m \in \N$. Define subgroups $K_{n,m} \subseteq K_{n, 0}$ by \[ K_{n,m} := K_{n, 0} \cap \nu^{-1}(1 + p^m \widehat{\Z}) = \{ g \in K_{n, 0} : \nu(g) \equiv 1 \text{ (mod } p^m \widehat{\Z})  \}.\]
\end{definition}

\begin{remark} As explained in \cite[5.4]{LSZ1}, if $U\subseteq {\G}(\Af)$ is an open compact subgroup such that \[ \nu(U)\cdot (1+p^m \widehat{\Z}) = \widehat{\Z}^\times,\] then there is an isomorphism of $\Q$-schemes
\[ \Sh_{\G}(U \cap \nu^{-1}(1 + p^m \widehat{\Z})) \simeq  \Sh_{\G}(U) \times_{\operatorname{Spec}(\Q)} \operatorname{Spec}(\Q(\zeta_{p^m})),\]
which intertwines the action of $g \in {\G}(\A_f)$ on the left-hand side with the one of $(g,\sigma_g)$ on the right-hand side, where $\sigma_g=\operatorname{Art}(\nu(g)^{-1})|_{\Q(\zeta_{p^m})}$. In particular, we have
\[ \Sh_{\G}(K_{n,m}) \simeq  \Sh_{\G}(K_{n, 0}) \times_{\operatorname{Spec}(\Q)} \operatorname{Spec}(\Q(\zeta_{p^m})).\]
\end{remark}

\subsection{The classes at level $K_{n,m}$}

For two given integers $n,m \geq 1$, take $n' = n + 3(m+1)$ and define the projection
\[t_m: \Sh_{\G}(K_{n',m}') \to \Sh_{\G}(K_{n,m}),\]
induced by right multiplication by the element  $ \eta_p^m = diag(p^{3m}, p^{2m}, p^{2m}, p^m, p^m, 1) \in \G(\qp)$ defined in \S \ref{levelstructures}.

\begin{remark}
The map $t_m$ is well-defined. Indeed, we need to check that $\eta_p^{-m} K'_{n', m} \eta_p^{m} \subseteq K_{n, m}$. Recall that $K'_{n', m} = K_{n'} \cap \eta_p^m K_{n'} \eta_p^{-m} \cap K_\G(p^m)$ and $K_{n, m} = K_{n, 0} \cap \nu^{-1}(1 + p^m \widehat{\Z})$, so we have $\eta_p^{-m} K'_{n', m} \eta_p^{m} = \eta_p^{-m} K_{n'} \eta_p^m \cap K_{n'} \cap \eta_p^{-m} K_\G(p^m) \eta_p^m$. This is obviously contained in $K_n$ and in $\nu^{-1}(1 + p^m \widehat{\Z})$. Finally, if $g \in K_{n'} \cap \eta_p^{-m} K_\G(p^m)\eta_p^m$, it satisfies the extra conditions modulo $p$ imposed in the definition of $K_{n, 0}$.
\end{remark}

Before defining the classes we note that the push-forward by $ t_{m,*}$ makes sense with our $p$-adic integral coefficients. 

\begin{lemma}\label{lemmaforpushforward}
There is a well defined action of $\eta_p^{-1} / p^{\lambda_2 + \lambda_3}$ on $W^\lambda_\zp$ defining a morphism of sheaves \[ t_{m, \flat}^\lambda : \mathcal{W}^\lambda_\zp \to t_m^* ( \mathcal{W}^\lambda_\zp ). \] In particular, we have a map
\[ t_{m, *}^\lambda : H^7_\et(\Sh_{\G}(K'_{n',m}), \mathcal{W}^\lambda_\zp(4+\tfrac{k-|\lambda|}{2})) \to  H^7_\et(\Sh_{\G}(K_{n,m}), \mathcal{W}^\lambda_\zp(4+\tfrac{k-|\lambda|}{2})), \]
defined by composing the map in cohomology induced by $t_{m, \flat}^\lambda$ with the trace of $t_m$ in \'etale cohomology.
\end{lemma}

\begin{proof}

We need to show that the matrix $\eta^{-1}_p = diag(p^{-3}, p^{-2}, p^{-2}, p^{-1}, p^{-1}, 1)$ acts on $W^\lambda_\zp$ and that its image is contained in $p^{\lambda_2 + \lambda_3} W^\lambda_\zp$.
Let $S$ be the one dimensional split torus $diag (x^3, x^2, x^2, x, x, 1)$ of $\G$. Then $V^\lambda$ decomposes as the direct sum of its weight spaces relative to $S$, with weights between $0$ and $3 \lambda_1 + 2 \lambda_2 + 2 \lambda_3$. We deduce that $S$ acts on the highest weight subspace of $W^\lambda = V^\lambda \otimes \nu^{-|\lambda|}$ through the character $diag (x^3, x^2, x^2, x, x, 1) \mapsto x^{-( \lambda_2 + \lambda_3 )}$ and, in particular, the action of $\eta_p^{-1}$ on every $S$-weight space (and hence on all $W^\lambda$) will be divisible by $p^{\lambda_2 + \lambda_3}$, thus showing the claim. 
\end{proof}

\begin{remark}\label{onthenormoftm}
 Observe that the normalisation by $p^{-(\lambda_2 + \lambda_3)}$ is such that the action of $p^{-(\lambda_2 + \lambda_3)} \eta_p^{-1}$ on the $S$-highest weight subspace of $W^\lambda$ is trivial and divisible by $p$ elsewhere. This optimal normalisation of the map $t_{m, *}^\lambda$ will be very helpful (in a rather subtle way) when defining our cohomology classes at integral level and proving their norm relations (cf. Theorem \ref{vernormrel}).
\end{remark}

We are now ready to define the following.

\begin{definition} \leavevmode 
\begin{itemize}
\item Let $\mathscr{Z}^{[\lambda, \mu]}_{n,m} := t_{m,*}^\lambda( \tilde{\mathscr{Z}}^{[\lambda, \mu]}_{n',m}) \in H^{7}_{\operatorname{mot}}(\Sh_{\G}(K_{n,m}), \mathscr{W}^{\lambda}_{\Q}(4+\tfrac{k-|\lambda|}{2}))$.
\item Let $_cz^{[\lambda,\mu]}_{n,m}$ be the class 
\[ t_{m,*}^\lambda({}_c\tilde{z}^{[\lambda, \mu]}_{n',m}) \in H^{7}_{\et}(\Sh_{\G}(K_{n,m}), \mathcal{W}^{\lambda}_{\zp}(4+\tfrac{k-|\lambda|}{2})).\] 
\end{itemize}
\end{definition}

\subsection{Norm relations at $p$: varying the level}

We now show that the various classes that we constructed are compatible when we vary the variable $n$. Denote by
\[ \pi_{1,n}: \Sh_{\GL_2}(K_1(n + 1 )) \to \Sh_{\GL_2}(K_1(n)) \]
the natural projection map. We recall the following standard result.

\begin{lemma}\label{normreleisensteinclass} 
We have
\[ (\pi_{1, n})_* ({}\operatorname{_cEis}^k_{\et, n + 1 }) = \begin{cases} {}\operatorname{_cEis}^k_{\et, n  } & \text{ if } n \geq 1, \\ 
(1 - p^k d_p^*){} \operatorname{_cEis}^k_{\et, n  } & \text{ if } n=0, \end{cases} \]
where $d_p \in \GL_2(\widehat{\Z})$ denotes any element congruent to $\left( \begin{smallmatrix} 1 & \\ & p \end{smallmatrix} \right)$ modulo $N$, for  $N \geq 1$ coprime to $p$ being the order of the auxiliary torsion section $x$ of Definition \ref{defeisclassesplusfootnote}.
\end{lemma}

\begin{proof}
This is well-known (e.g. \cite[Theorem 4.3.3]{KLZ} or \cite{scholl}) and it basically follows from the compatibility relations that Siegel units satisfy, the only difference being the action the push-forward of the multiplication by $p$ map on the coefficient sheaf $\mathcal{H}^k_{\zp}$ given by multiplication by $p^k$.
\end{proof}

Let
\[ pr_{n}: \Sh_{\G}(K_{n+1, m}') \to \Sh_{\G}(K_{n, m}') \]
denote the natural projection map.

\begin{proposition} \label{normrelationsn}
We have 
\[ pr_{n, *} ({}_c \tilde{z}_{n+1,m}^{ [ \lambda, \mu] })= \begin{cases} {}_c\tilde{z}_{n,m}^{ [ \lambda, \mu] } & \text{ if } n\geq 1, \\ 
                                                                                 (1 - p^{k} D_p^* ){}_c\tilde{z}_{n,m}^{ [ \lambda, \mu] }  & \text{ if }  n=0, \end{cases} \]
where $D_p \in \H(\widehat{\Z}) \subseteq \G(\widehat{\Z})$ is any matrix whose first $\GL_2$-component is congruent to ${\matrix 1 0 0 p}$ modulo $N$, for any $N$ as in Lemma \ref{normreleisensteinclass}.
\end{proposition}

\begin{proof}

From commutativity of the diagram 
\[ \xymatrix{ \Sh_{\H}(u K_{n+1,m}' u^{-1} \cap \H) \ar[rr]^-{ \iota_{K_{n+1, m}'}^{u} } \ar[d]_{\tilde{pr}_n} & & \Sh_{\G}(K_{n+1,m}') \ar[d]^{pr_n} \\  
                    \Sh_{\H}(u K_{n,m}' u^{-1} \cap \H) \ar[rr]^-{\iota_{K_{n, m}'}^{u} } & & \Sh_{\G}(K_{n,m}'),   }
\]
where $ \tilde{pr}_n: \Sh_{\H}(u K_{n+1,m}' u^{-1} \cap \H) \to \Sh_{\H}(u K_{n,m}' u^{-1} \cap \H)$ is the natural projection map, we obtain 

\begin{eqnarray*}
pr_{n, *} ({}_c \tilde{z}_{n+1,m}^{ [ \lambda, \mu ] }) &:=&  pr_{n, *} \circ \iota_{K_{n + 1,m}', u, *}^{[\lambda,\mu]}  ({}_c z^{k}_{\H, n + 1, m} ) \\ 
&=&   \iota_{K_{n ,m}', u, *}^{[\lambda,\mu]} \circ \tilde{pr}_{n, *} ({}_c z^{k}_{\H, n + 1, m} ),
\end{eqnarray*}
where we have denoted by ${}_c z^{k}_{\H, n, m}$ the class $\operatorname{pr}_{1,n,m}^* ({}\operatorname{_cEis}^{k}_{\et , n})$.
Thus, we are reduced to studying compatibility relations of ${}_c z^{k}_{\H, n + 1, m}$ with respect to $\tilde{pr}_{n, *}$, which follows from Lemma \ref{normreleisensteinclass}. Indeed, the Cartesian diagram
\[ \xymatrix{ 
\Sh_{\H}( u K_{n+1,m}' u^{-1} \cap \H) \ar[d]_{\tilde{pr}_{n}} \ar[rr]^-{\operatorname{pr}_{1,n+ 1,m}} & & \Sh_{\GL_2}( K_1({n + 1 } )) \ar[d]^{\pi_{1,n}} \\
 \Sh_{\H}( u K_{n,m}' u^{-1} \cap \H ) \ar[rr]^-{\operatorname{pr}_{1,n,m} } & & \Sh_{\GL_2}( K_1(n   ) ) } \]
gives 
\begin{eqnarray*} 
\tilde{pr}_{n, *} ({}_c z^{k}_{\H, n + 1, m}) &=&  \tilde{pr}_{n, *} \circ \operatorname{pr}_{1,n+1,m}^* ({}\operatorname{_cEis}^{k}_{\et , n+1}) \\
&=& \operatorname{pr}_{1,n,m}^* \circ (\pi_{1, n})_*({}\operatorname{_cEis}^{k}_{\et , n+1}).
\end{eqnarray*}

Thus, the result now follows from Lemma \ref{normreleisensteinclass} and  by observing that, for any $D_p \in \H(\widehat{\Z})$ as in the statement of the proposition, one has a commutative diagram
\[ \xymatrix{ 
 \Sh_{\H}( u K_{n,m}' u^{-1} \cap \H ) \ar[d]_{D_p}  \ar[rr]^-{\operatorname{pr}_{1,n,m} } & & \Sh_{\GL_2}( K_1(n   ) )  \ar[d]^{ d_p} \\
\Sh_{\H}( u K_{n,m}' u^{-1} \cap \H ) \ar[rr]^-{\operatorname{pr}_{1,n,m} } & & \Sh_{\GL_2}( K_1(n   ) ) ,}\]
 which can be easily checked using the moduli space description of the varieties.
\end{proof}

This immediately translates into the identical norm relations for the level $K_{n, m}$ classes.  

\begin{corollary} Let $pr_{n}: \Sh_{\G}(K_{n+1, m}) \to \Sh_{\G}(K_{n, m})$ be the natural projection map. We have 
\[ pr_{n, *} ({}_cz_{n+1,m}^{ [ \lambda, \mu] })= \begin{cases} {}_c z_{n,m}^{ [ \lambda, \mu] } & \text{ if } n\geq 1, \\ 
                                                                                 (1 - p^{k} D_p^* ){}_c z_{n,m}^{ [ \lambda, \mu] }  & \text{ if }  n=0, \end{cases} \]
where $D_p \in \H(\widehat{\Z}) \subseteq \G(\widehat{\Z})$ is any matrix whose first $\GL_2$-component is congruent to ${\matrix 1 0 0 p}$ modulo $N$, for any $N$ as in Lemma \ref{normreleisensteinclass}.
\end{corollary}

\subsection{Norm relations at $p$: cyclotomic variation}

In the section we prove our main result stating that our cohomology classes satisfy the Euler system relations at powers of $p$.

\subsubsection{Hecke operators}

We now define the (normalised) Hecke operator which is going to show up in the norm compatibility relations of our cohomology classes.

\begin{definition} \label{heckeopascor} We define the Hecke operator $\mathcal{U}_{p}'$ acting on $H^7_{\et}(\Sh_{\G}(K'_{n,m}), \mathcal{W}^\lambda_\zp(4+\tfrac{k-|\lambda|}{2}))$ to be the action of $p^{-(\lambda_2 + \lambda_3)} \cdot K_{n,m}' \eta_{p}^{-1} K_{n,m}'$, where $K_{n,m}' \eta_{p}^{-1} K_{n,m}'$ is seen as an element of the Hecke algebra $\mathcal{H}(K'_{n, m} \backslash G(\Af) / K'_{n, m})_\zp$ of $K'_{n, m}$-bi-invariant smooth compactly supported $\zp$-valued functions on $\G(\Af)$. In other words, the action of $K_{n,m}' \eta_{p}^{-1} K_{n,m}'$ on cohomology is the one induced from the following correspondence on $\Sh_{\G}$:
\[ \xymatrix{  \Sh_{\G}(K_{n,m(p)}')  \ar[d]_{\pi_p'} \ar[dr]^{\eta_p} & \\ \Sh_{\G}(K_{n,m}')  \ar@{.>}[r] &  \Sh_{\G}(K_{n,m}'),}  \]
where the vertical arrow is the natural projection $\pi_p'$ and the diagonal one is induced by right multiplication by $\eta_p$, 
and hence $\mathcal{U}_{p}'$ is given by the composition
\[ H^7_{\et}(\Sh_{\G}(K'_{n,m}), \mathcal{W}^\lambda_\zp(4+\tfrac{k-|\lambda|}{2})) \xrightarrow{(\pi_p')^*} H^7_{\et}(\Sh_{\G}(K'_{n,m(p)}), \mathcal{W}^\lambda_\zp(4+\tfrac{k-|\lambda|}{2})) \xrightarrow{\eta_{p, *}^\lambda} H^7_{\et}(\Sh_{\G}(K'_{n,m}), \mathcal{W}^\lambda_\zp(4+\tfrac{k-|\lambda|}{2})), \] where $\eta_{p, *}^\lambda$ is the normalised map defined exactly in the same way as the map $t_{m, *}^\lambda$ of Lemma \ref{lemmaforpushforward}.

\end{definition}

\begin{remark}
 The notation chosen for the Hecke operator is motivated by the fact that $\mathcal{U}_{p}'$ is dual to the Hecke operator associated to $\eta_p$.
\end{remark}

\subsubsection{Norm relation for the classes ${}_c \tilde{z}_{n,m}^{ [ \lambda, \mu ] }$}

Recall that the diagonal matrix $\eta_p : = (p^3, p^2, p^2, p, p, 1) \in \G(\qp)$ induces a morphism of Shimura varieties $\eta_p : \Sh_{\G}(K'_{n,m(p)}) \to \Sh_{\G}(K'_{n,m})$ and, by Lemma \ref{lemmaforpushforward}, a map 
\[ \eta_{p,*}^\lambda : H^7_{\et}(\Sh_{\G}(K'_{n,m(p)} ), \mathcal{W}^\lambda_\zp(4+\tfrac{k-|\lambda|}{2})) \to  H^7_{\et}(\Sh_{\G}(K'_{n,m}), \mathcal{W}^\lambda_\zp(4+\tfrac{k-|\lambda|}{2})).  \]

Let $m \geq 1$, $n\geq 3(m+1)$, and denote by $\tilde{\eta}_p$ the composition of the natural projection map $pr: \Sh_{\G}(K'_{n,m+1}) \to \Sh_{\G}(K'_{n,m(p)})$ with the map $\eta_p : \Sh_{\G}(K'_{n,m(p)}) \to \Sh_{\G}(K'_{n,m})$. By the same arguments as in Lemma \ref{lemmaforpushforward}, we can once more define a normalised trace
\[ \tilde{\eta}_{p,*}^\lambda : H^7_{\et}(\Sh_{\G}(K'_{n,m + 1} ), \mathcal{W}^\lambda_\zp(4+\tfrac{k-|\lambda|}{2})) \to  H^7_{\et}(\Sh_{\G}(K'_{n,m}), \mathcal{W}^\lambda_\zp(4+\tfrac{k-|\lambda|}{2})), \]
as the composition the trace of $pr$ with $\eta_{p, *}^\lambda$.

We have the following push-forward compatibility relation.

\begin{theorem} \label{cyclnormrelation}
For $m \geq 1$, $n\geq 3(m+1)$, we have 
\[ \tilde{\eta}_{p,*}^\lambda ( {}_c\tilde{z}_{n,m+1}^{ [ \lambda, \mu] } ) = \mathcal{U}'_p \cdot {}_c\tilde{z}_{n,m}^{ [ \lambda, \mu] }, \]
where $\mathcal{U}'_p$ is the Hecke operator defined in Definition \ref{heckeopascor}.
\end{theorem}

\begin{proof}
 Denote by ${}_c z^{k}_{\H, n, m}$ the class $\operatorname{pr}_{1,n,m}^* ({}\operatorname{_cEis}^{k}_{\et , n})$. 
 The result follows from Lemma \ref{cartdiag}. Indeed, by the definition of the class ${}_c\tilde{z}_{n,m+1}^{ [ \lambda, \mu] }$ we have
\begin{align*}
pr_*( {}_c\tilde{z}_{n,m+1}^{ [ \lambda, \mu] } ) &= pr_* \circ \iota_{K_{n,m + 1}', u, *}^{[\lambda,\mu]}  ({}_c z^{k}_{\H, n, m + 1} ) \\
&= pr_* \circ \iota_{K_{n,m + 1}', u, *}^{[\lambda,\mu]}  \circ \pi_p^* ({}_c z^{k}_{\H, n, m} ),
\end{align*}
where $\pi_p$ is as in Lemma \ref{cartdiag}. By the Cartesianness of the square of the diagram of Lemma \ref{cartdiag}, we have that
\[pr_* \circ \iota_{K_{n,m + 1}', u, *}^{[\lambda,\mu]}  \circ \pi_p^* =  (\pi'_p)^* \circ \iota_{K_{n,m}', u, *}^{[\lambda,\mu]}, \]
so we deduce
\[ pr_*( {}_c\tilde{z}_{n,m+1}^{ [ \lambda, \mu] } ) = (\pi'_p)^* \circ \iota_{K_{n,m}', u, *}^{[\lambda,\mu]} ({}_c z^{k}_{\H, n, m}) = (\pi'_p)^*({}_c\tilde{z}_{n,m}^{ [ \lambda, \mu ] }), \] where the last equality follows by definition. Hence, by applying $\eta_{p,*}^\lambda$ to both sides, we get
\[ \tilde{\eta}_{p,*}^\lambda ( {}_c\tilde{z}_{n,m+1}^{ [ \lambda, \mu] } ) = \eta_{p,*}^\lambda \circ (\pi'_p)^*({}_c \tilde{z}_{n,m}^{ [ \lambda, \mu] }) = \mathcal{U}'_p \cdot {}_c\tilde{z}_{n,m}^{ [ \lambda, \mu] }\]
as desired.
\end{proof}

\subsubsection{Norm relation for the classes ${}_c z_{n,m}^{ [ \lambda, \mu ] }$}

Call $\mathrm{norm}^{\Q(\zeta_{p^{m+1}})}_{\Q(\zeta_{p^{m}})}$ the norm map of the natural projection $\Sh_\G(K_{n,0})_{/\Q(\zeta_{p^{m+1}})} \to \Sh_\G(K_{n, 0})_{/\Q(\zeta_{p^{m}})}$. Moreover, let $\sigma_p$ denotes the image of $\tfrac{1}{p}\in \Q_p^*$ under the Artin reciprocity map.

\begin{theorem}\label{vernormrel}
For $n, m \geq 1$, we have
\[ \mathrm{norm}^{\Q(\zeta_{p^{m+1}})}_{\Q(\zeta_{p^{m}})} ( {}_c z_{n,m+1}^{ [ \lambda, \mu ] } ) = \tfrac{\mathcal{U}_p'}{\sigma_p^3} \cdot {}_c z_{n,m}^{ [ \lambda, \mu  ] }, \]
where $\mathcal{U}_p'$ is the Hecke operator associated to $p^{-(\lambda_2 + \lambda_3)} \cdot K_{n,m} \eta_{p}^{-1} K_{n,m}$.
\end{theorem}

\begin{proof}
We first deduce the norm relation at levels $K_{n,m}$.
 By Theorem \ref{cyclnormrelation} and the commutative diagram
 \[ \xymatrix{  \Sh_\G(K'_{n, m + 1}) \ar[r]^{\eta_p^{m + 1}} \ar[d]_{\tilde{\eta}_p} & \Sh_\G(K_{n, m + 1}) \ar[d]  \\
\Sh_\G(K'_{n, m}) \ar[r]^-{ \eta_p^m} & \Sh_\G(K_{n, m}),  } \]
where the right vertical arrow is the natural projection map, it suffices to show that the Hecke operator $\mathcal{U}'_p$ commutes with $t_{m, *}^\lambda$, i.e. that we have a commutative diagram
 \[ \xymatrix{  H^7(\Sh_\G(K'_{n', m }), \mathcal{W}_\zp^\lambda(4+\tfrac{k-|\lambda|}{2})) \ar[r]^{t_{m , *}^\lambda} \ar[d]_{\mathcal{U}'_p} & H^7(\Sh_\G(K_{n, m }), \mathcal{W}_\zp^\lambda(4+\tfrac{k-|\lambda|}{2})) \ar[d]^{\mathcal{U}'_p}  \\
H^7(\Sh_\G(K'_{n', m}), \mathcal{W}_\zp^\lambda(4+\tfrac{k-|\lambda|}{2})) \ar[r]^-{t_{m, *}^\lambda} & H^7(\Sh_\G(K_{n, m}), \mathcal{W}_\zp^\lambda(4+\tfrac{k-|\lambda|}{2})).  } \]

Recall that the Hecke operator $\mathcal{U}_p'$ at level $K_{n, m}$ is defined as the correspondence $pr_{2,*} \circ \eta_{p,*}^\lambda \circ pr_1^*$, where $pr_1$, $pr_2$ are natural projections sitting in the diagram
\[ \Sh_\G(K_{n, m}) \xleftarrow{pr_1} \Sh_\G(K_{n, m} \cap \eta_p K_{n, m} \eta_p^{-1}) \xrightarrow{\eta_p} \Sh_\G(\eta_p^{-1} K_{n, m} \eta_p \cap K_{n, m} ) \xrightarrow{pr_2} \Sh_\G(K_{n, m}). \]

Then, the two Hecke operators commute if $| K_{n,m} \cap \eta_p^{-1} K_{n,m} \eta_p \backslash K_{n,m} | = | K_{n',m}' \cap \eta_p^{-1} K_{n',m}' \eta_p \backslash K_{n',m}' |$. This is indeed the case, since both sizes can be checked to be $p^{12}$. Here, we are making an essential use of the extra congruences modulo $p$ satisfied by the elements in $K_{n,0}$.
Finally, the result follows after using the isomorphism
\[ \Sh_{\G}(K_{n,m}) \simeq  \Sh_{\G}(K_{n, 0}) \times_{\operatorname{Spec}(\Q)} \operatorname{Spec}(\Q(\zeta_{p^m})),\]
which intertwines the Hecke operator $\mathcal{U}'_p$ on the cohomology $\Sh_{\G}(K_{n,m})$ with $\operatorname{Art}(\nu(\eta_p))|_{\Q(\zeta_{p^m})} \mathcal{U}'_p= \sigma_p^{-3} \mathcal{U}'_p$.

\end{proof}

\begin{remark}
 For calculating the size of the quotient for $K_{n, m}$, one actually crucially uses the congruences modulo $p$ appearing in the definition of the level group $K_{n,0}$, and the result would not hold if we didn't impose those congruences.
\end{remark}

\section{$p$-adic interpolation}\label{interpolation}

We show in this section how the compatible systems of elements so far constructed can be $p$-adically interpolated. In \S \ref{interpweight}, we vary one of the variables of the weight to interpolate several classes geometrically constructed. In \S \ref{universalclass}, we construct a universal class that specialises, under some twisted moment maps, to all the geometrically constructed classes. The problem with this last approach is that we do not know how to naturally interpret these twisted moment maps. As an application, we can consider different specialisations of this universal class and thus obtain new classes at finite level that do not a priori come from a geometric construction.

In this section, we replace the Shimura varieties for $G \in \{ \GL_2, \H, \G \}$ with their smooth integral models defined over $\Z[\frac{1}{S}]$, for $S$ sufficiently large and divisible by $p$.

\subsection{$p$-adic interpolation for $\GL_2$}

We first recall the $p$-adic interpolation results that we need from \cite{kings}; we follow the exposition given by \cite{LSZ1}. Fix an open compact subgroup $U^{(p)} \subseteq \GL_2(\Af^{(p)})$ and let $U_n = U^{(p)} U^p_{n}$ with $U^p_{n} = K_1^p(n)$ (cf. \S \ref{levelstructures}) and assume that $U_n$ is sufficiently small for every $n$.

\begin{definition}
 We define
 \[ H^i_\et(\mathrm{Sh}_{\GL_2}(U_{\infty}), \zp(1)) := \varprojlim_{n \geq 1} H^i_{\et}(\mathrm{Sh}_{\GL_2}(U_{n}), \zp(1)), \]
 where the inverse limit is taken with respect to the natural trace maps.
\end{definition}

Recall that, for every $k \geq 0$, and $n \in \N$ we have moment maps
\[ \mathrm{mom}_{\GL_2, n}^k: H^1_\et(\Sh_{\GL_2}(U_{\infty}), \zp(1)) \to H^1_\et(\Sh_{\GL_2}(U_n), \mathscr{H}^k_\zp(1)), \]
defined in \cite[Theorem 4.5.1(2)]{KLZ}: let $e,f$ be a basis for the standard representation of $\GL_2$. If $\mathscr{H}^k_{\Z/p^n\Z}$ is the reduction modulo $p^n$ of $\mathscr{H}^k_\zp$, we have a section \[ e_n^k \in H^0_\et(\Sh_{\GL_2}(U_n) , \mathscr{H}^k_{\Z/p^n\Z} ) \] given by the reduction modulo $p^n$ of $e^{\otimes k}$ in $\operatorname{Sym}^k_\Z \otimes \operatorname{det}^{-k}$. Then, the moment map $\mathrm{mom}_{\GL_2, n}^k$ is defined by sending any element $(v_s)_{s \geq 1}$ to \[ (\operatorname{pr}^{U_{s}}_{U_{n}})_*(v_s \cup e_s^k)_{s \geq n} \in \varprojlim_{s \geq n} H^1_\et(\Sh_{\GL_2}(U_n), \mathscr{H}^k_{\Z/p^s\Z}(1))=H^1_\et(\Sh_{\GL_2}(U_n), \mathscr{H}^k_\zp(1)).\]
Crucially, we have the following interpolation result.

\begin{theorem} [\cite{kings}] \label{gl2interpolation} The class $_c\mathbf{Eis}_{\GL_2} := (\partial({}_c g_{n}))_{n \geq 1} \in H^1_\et(\Sh_{\GL_2}(U_{\infty}), \zp(1))$ is such that, for every $k, n \in \N$, we have
\[ \mathrm{mom}_{\GL_2, n}^k(_c \mathbf{Eis}_{\GL_2}) ={}\operatorname{_cEis}^k_{\et,n}. \]
\end{theorem}

\subsection{$p$-adic interpolation for $\G$: varying the weight}

In \S \ref{momentmapsforG}, we discuss interpolation properties of the classes at level $K_{n,m}$, which result in a compatibility with respect to varying the weight of our local systems in one direction. This reflects the asymmetry of the construction of our cohomology classes. Proposition \ref{interpG} below is not subject to any $\mathcal{U}_p'$-ordinarity assumption and it is a direct consequence of Proposition \ref{normrelationsn} and Theorem \ref{gl2interpolation}.

\subsubsection{Explicit branching laws} \label{explicitbl}

It will be useful to construct explicit highest weight vectors of the sub-$\H$-spaces of $V^\lambda$ given by the branching laws. Only the element $W$ defined below will be used in \S \ref{momentmapsforG} - \S \ref{interpweight} and the rest of them will only be useful for \S \ref{universalclass}, so the reader is invited to come back to their definition when needed.

Let us first write down the decomposition of the three basic representations $V^{(1 \geq 0 \geq 0)}, V^{(1 \geq 1 \geq 0)}$ and $V^{(1 \geq 1 \geq 1)}$ as a direct sum of $\H'$-representations. Applying the branching laws we readily get the following decompositions of $\H'$-representations:

\[ V^{(1 \geq 0 \geq 0)} = ( V^{(0 \geq 0)} \boxtimes \operatorname{Sym}^1) \oplus (V^{(1 \geq 0)} \boxtimes \operatorname{Sym}^0), \]
\[ V^{(1 \geq 1 \geq 0)} = [(V^{(0 \geq 0)} \boxtimes \operatorname{Sym}^0)\otimes \nu ]\oplus (V^{(1 \geq 0)} \boxtimes \operatorname{Sym}^1) \oplus (V^{(1 \geq 1)} \boxtimes \operatorname{Sym}^0), \]
\[ V^{(1 \geq 1 \geq 1)} = [(V^{(1 \geq 0)} \boxtimes \operatorname{Sym}^0)\otimes \nu ] \oplus (V^{(1 \geq 1)} \boxtimes \operatorname{Sym}^1). \]
Let $V$ be the standard representation of $\G$ with its symplectic basis $\langle e_1, e_2, e_3, f_3, f_2, f_1 \rangle$. Since we will only be interested in those $\H$-factors of the form $\operatorname{Sym}^{(k, 0, 0)}$, using the branching laws from $\H'$ to $\H$ in the decompositions above, we fix highest weight vectors for the following $\H$-representations:
\[ W := e_1 \in \operatorname{Sym}^{(1, 0, 0)} \subseteq V^{(1 \geq 0)} \boxtimes \operatorname{Sym}^0 \subseteq V^{(1 \geq 0 \geq 0)}, \]
\[ X := e_1 \wedge f_1 - e_2 \wedge f_2 \in \operatorname{Sym}^{(0, 0, 0)}\otimes \operatorname{det} \subseteq (V^{(0 \geq 0)} \boxtimes \operatorname{Sym}^0)\otimes \nu \subseteq V^{(1 \geq 1 \geq 0)}, \]
\[ Y := e_2 \wedge f_2 - e_3 \wedge f_3 \in \operatorname{Sym}^{(0, 0, 0)}\otimes \operatorname{det} \subseteq V^{(1 \geq 1)} \boxtimes \operatorname{Sym}^0 \subseteq V^{(1 \geq 1 \geq 0)}, \]
\[ Z := e_1 \wedge e_2 \wedge f_2 - e_1 \wedge e_3 \wedge f_3 \in \operatorname{Sym}^{(1, 0, 0)}\otimes \operatorname{det} \subseteq (V^{(1 \geq 0)} \boxtimes \operatorname{Sym}^0)\otimes \nu \subseteq V^{(1 \geq 1 \geq 1)}. \] 
Observe that, for $p, q, r, s \in \N$, we have
\[ W^p \cdot X^q \cdot Y^r \cdot Z^s \in \operatorname{Sym}^{(p + s, 0, 0)}\otimes \operatorname{det}^{q+r+s} \subseteq (V^{(p + r + s \geq r)} \boxtimes \operatorname{Sym}^0)\otimes \nu^{q+s} \subseteq V^{(p + q + r + s \geq q + r + s \geq s)}, \]
 where the operation $\cdot$ denotes Cartan product.

\begin{lemma} \label{choiceofsomehighestwv}
 Let $\lambda = (\lambda_1 \geq \lambda_2 \geq \lambda_3)$ and let $\mu = (\mu_1 \geq \mu_2) \in \mathcal{A}(\lambda)$ be such that $k := \mu_1 - \mu_2 = \lambda_1 - \lambda_2 + \lambda_3$. Then
 \[ v^{[\lambda, \mu]}:= W^{\mu_1 - \mu_2 - \lambda_3} \cdot X^{\lambda_2 - \lambda_3 - \mu_2} \cdot Y^{\mu_2} \cdot Z^{\lambda_3} \in \operatorname{Sym}^{(k, 0, 0)} \otimes \operatorname{det}^{\lambda_2} \subseteq (V^\mu \boxtimes \operatorname{Sym}^0) \otimes \nu^{\lambda_2-\mu_2} \subseteq V^\lambda \]
is an $\H$-highest weight vector.
\end{lemma}

\begin{proof}
The vector $v^{[\lambda, \mu]}$ of the statement is an $\H$-highest weight vector, because it is a Cartan product of $\H$-highest weight vectors.
\end{proof}

Let $\mu$ and $k$ be as in Lemma \ref{choiceofsomehighestwv}. Let $w^{[\lambda, \mu]}$ be the image of $v^{[\lambda, \mu]}$ in $W^\lambda_{\Z}=V^\lambda_{\Z} \otimes \nu^{-|\lambda|}$. We have the $\H$-equivariant inclusion
\[ \mathrm{br}^{[\lambda, \mu]} : \operatorname{Sym}^{(k,0,0)}_\Z \otimes \operatorname{det}^{-k} \subseteq W^\lambda_{\Z} \otimes \nu^{\lambda_2}.\]
This is defined by sending the image of the vector $e_1^k \in \operatorname{Sym}^{(k,0,0)}_\Z$ in $\operatorname{Sym}^{(k,0,0)}_\Z \otimes \operatorname{det}^{-k}$ to $w^{[\lambda, \mu]} \otimes \zeta^{\lambda_2}$, where $\zeta$ is a basis of the symplectic multiplier representation $\nu$. Now, for any integer $t \geq 1$, denote by $w^{[\lambda, \mu]}_t$ and $\zeta_t$ their reduction modulo $p^t$.

\subsubsection{Moment maps} \label{momentmapsforG}

Consider the level groups $K_{n,m} \subset \G(\Af)$ defined in \S \ref{levelstructures}. Let us begin by describing the moment maps from \[ H^i_\et(\mathrm{Sh}_{\G}(K_{\infty,m}), \mathcal{L}) := \varprojlim_{n \geq 1} H^i_{\et}(\mathrm{Sh}_{\G}(K_{n,m}), \mathcal{L}), \] where the inverse limit is taken with respect to the natural trace maps, to $H^i_{\et}(\mathrm{Sh}_{\G}(K_{n,m}), \mathcal{L}')$ for certain $\Z_p$-local systems $\mathcal{L}$, $\mathcal{L}'$.

\begin{lemma} \label{momG}
 Let $\lambda = (\lambda_1 \geq \lambda_2 \geq \lambda_3)$ and $\kappa= (k \geq 0 \geq 0)$. For $n \geq 1$, we have a map
 \begin{align*}
  \mathrm{mom}^{k}_{{\G}, n}: H^i_\et(\mathrm{Sh}_{\G}(K_{\infty,m}), \mathcal{W}^\lambda_{\zp}) &\xrightarrow{\sim} \varprojlim_{s} H^i_\et(\mathrm{Sh}_\G(K_{s, m}),\mathcal{W}^{\lambda}_{\Z / p^s}) \\
  &\to \varprojlim_{s} H^i_\et(\mathrm{Sh}_\G(K_{s, m}),\mathcal{W}^{\lambda}_{\Z / p^s}\otimes \mathcal{W}^{\kappa}_{\Z / p^s}) ) \\
  &\to \varprojlim_{s} H^i_\et(\mathrm{Sh}_\G(K_{s, m}), \mathcal{W}^{\lambda + \kappa}_{\Z / p^s} ) \\
  &\to \varprojlim_{s} H^i_\et(\mathrm{Sh}_\G(K_{n, m}), \mathcal{W}^{\lambda + \kappa}_{\Z / p^s}) \\
  &\xrightarrow{\sim} H^i_\et(\mathrm{Sh}_\G(K_{n, m}), \mathcal{W}^{\lambda + \kappa}_\zp),
 \end{align*}
where the first map is an isomorphism, the second map is obtained by taking the cup product against the \'etale section associated to $w^{[\kappa, \mu]}_s$ in $H^0_\et(\mathrm{Sh}_\G(K_{s, m}), \mathcal{W}^{\kappa}_{\Z / p^s})$ for $\mu=(k,0)$ as in Lemma \ref{choiceofsomehighestwv}, the third one is the map induced in cohomology by the Cartan product, the fourth one is obtained by taking the projection to level $n$, and the last one is an isomorphism again.
\end{lemma}

\begin{proof}
To show that the moment map is well defined, we need to verify that $w^{[\kappa, \mu]}_s$ is fixed by $K_{s,m}$.  
Notice that, by the chosen twist in the definition of $W^{\kappa}$, the vector $w^{[\kappa, \mu]}=W^k=e_1^k$ is fixed by the subgroup of the Levi factor of the Klingen parabolic given by \[ \left( \begin{smallmatrix} x & {} & {} \\ & A & \\  &   &  1 \end{smallmatrix} \right), \] where $A \in \GSp_4$ and $x=\nu(A)$. Finally, since $w^{[\kappa, \mu]}$ is fixed by the unipotent of the Borel, it is in particular fixed by the unipotent of the Klingen parabolic. We deduce then that  $w^{[\kappa, \mu]}_s$  is fixed by $K_{s,m}$ and defines an element in $H^0_\et(\mathrm{Sh}_\G(K_{s, m}), \mathcal{W}^{\kappa}_{\Z / p^s})$, hence the moment map is well defined.
\end{proof}

\subsubsection{$p$-adic interpolation} \label{interpweight}

We can now state the main theorem of the section.
 
\begin{proposition} \label{interpG}
There exists an element
 \[ _c\mathbf{z}_m \in H^7_\et(\mathrm{Sh}_{\G}(K_{\infty,m}), \zp(4)) \]
 such that, for all $n \geq 1$ and $k \geq 0$, we have
 \[ \mathrm{mom}^{k}_{{\G}, n}(_c \mathbf{z}_m) = {} _cz^{[\kappa, \mu]}_{\mathrm{\acute e t}, n , m} \]
 as elements of $H^i_{\et}(\mathrm{Sh}_{\G}(K_{n,m}), \mathcal{W}^{\kappa}_\zp ( 4  ) )$, where  $\kappa= (k \geq 0 \geq 0)$ and $\mu = (k \geq 0)$.
\end{proposition}

\begin{proof}
The proof is similar to \cite[Proposition 9.3.3]{LSZ1}.
First, notice that from the very definition of the moment map $\mathrm{mom}_{\GL_2, n}^k$, we can define
\[ \mathrm{mom}^k_{{\H}, n}: H^1_\et (\Sh_{\H}(uK'_{\infty,m}u^{-1} \cap \H), \zp(1)) \to H^1_\et(\Sh_{\H}(uK'_{n,m}u^{-1} \cap \H), \mathcal{H}^{(k, 0 , 0)}_\zp(1)), \] by taking the cup-product with the image $w_t^k$ of $e_1^{\otimes k}  \in \operatorname{Sym}^{(k, 0, 0)}_\Z$  in $\operatorname{Sym}^{(k, 0, 0)}_{\Z/p^t\Z} \otimes \operatorname{det}^{-k}$. Indeed, $w_t^k$ defines a section in $H^0_\et(\Sh_{\H}(uK'_{t,m}u^{-1} \cap \H), \mathcal{H}^{(k, 0 , 0)}_\zp)$, because $uK'_{t,m}u^{-1} \cap \H \subset uK_{t,0}u^{-1} \cap \H = K_{t,0} \cap \H$ fixes $w_t^k$. By Theorem \ref{gl2interpolation}, the class $_c \mathbf{z}_{\H,m}:= (pr_{ 1 , n, m }^*)_{n >>0}( {_c\mathbf{Eis}_{\GL_2}})$ satisfies \[ \mathrm{mom}_{\H, n}^k(_c \mathbf{z}_{\H,m}) = pr_{ 1 , n, m }^* ({}\operatorname{_cEis}^k_{\et,n}).\] 

Moreover, for $\kappa= (k \geq 0 \geq 0)$, $\mu = (k \geq 0)$, and $n'=n+3(m+1)$, we claim that we have a commutative diagram
\[ \xymatrix{  H^1_\et (\Sh_{\H}(uK'_{\infty,m}u^{-1} \cap \H), \zp(1))  \ar[r]^-{\iota_{\infty,u,*}}  \ar[d]_{\mathrm{mom}^k_{{\H}, n'}} & H^7_\et(\Sh_{\G}(K'_{\infty,m}), \zp(4 )) \ar[r]^-{t_{\infty ,m,*}}  & H^7_\et(\mathrm{Sh}_{\G}(K_{\infty,m}), \zp(4)) \ar[d]^{\mathrm{mom}^{k}_{{\G}, n}} \\
H^1_\et(\Sh_{\H}(uK'_{n',m}u^{-1} \cap \H), \mathcal{H}^{(k, 0 , 0)}_\zp(1)) \ar[r]^-{\iota^{[ \kappa, \mu ]}_{K'_{n',m},u,*}} & H^7_\et(\Sh_{\G}(K_{n', m}'), \mathcal{W}^\kappa_\zp(4)) \ar[r]^-{t_{m,*}^\kappa} & H^7_\et(\Sh_{\G}(K_{n, m}), \mathcal{W}^\kappa_\zp(4)) ,  }
 \]
where $\iota_{\infty,u, *}= ( \iota_{K'_{n,m},u,*})_{n >>0}$ and similarly $t_{\infty ,m,*}$ is the collection of pushforward maps of $t_m: \Sh_{\G}(K_{s',m}') \to \Sh_{\G}(K_{s,m}),$ for $s'=s+3(m+1)$, as $s \geq 1$ varies. Indeed, this follows from the explicit choice of the $\H$-highest weight vector $w^{[\kappa, \mu]}$, as we now explain. Recall that the branching map $\br^{[\kappa, \mu]}$ sends $e_1^{\otimes k}$ to $w^{[\kappa, \mu]}$. Notice that $u^{-1} \cdot w^{[\kappa, \mu]}=w^{[\kappa, \mu]}$ lies in the highest weight space relative to the one dimensional torus $S=diag (x^3, x^2, x^2, x, x, 1)$ of $\G$, since $S$ acts trivially on it. As discussed in Remark \ref{onthenormoftm}, the map $t_{m, \flat}^\kappa$ acts trivially on the $S$-highest weight subspace, thus we have that \[  t_{m, \flat}^\kappa( u_* w^{[\kappa, \mu]}_{n'})=t_{m, \flat}^\kappa( (u^{-1})^* w^{[\kappa, \mu]}_{n'})=  t_{m, \flat}^\kappa(w^{[\kappa, \mu]}_{n'}) = (t_{m})^*w^{[\kappa, \mu]}_{n}\] as sections of $(t_{m})^* \mathcal{W}^\kappa_{\Z/p^n \Z}$. This equality is crucially employed in the proof of the commutativity of the diagram above. Recall that $t_{m,*}^\kappa = t_{m,*} \circ t_{m,\flat}^\kappa$, then for $\mathbf{v}=(v_{s'})_{s \geq 1} \in H^1_\et (\Sh_{\H}(uK'_{\infty,m}u^{-1} \cap \H), \zp(1))$, 
\begin{align*} 
( t_{m,*}^\kappa \circ \iota^{[ \kappa, \mu ]}_{K'_{n',m},u,*} \circ \mathrm{mom}^k_{{\H}, n'})(\mathbf{v}) &= ( t_{m,*}^\kappa \circ \iota^{[ \kappa, \mu ]}_{K'_{n',m},u,*} ) \left( (\operatorname{pr}^{K_{s',m}'}_{K_{n',m}'})_* (v_{s'} \cup w^k_{s'}) \right)_{s \geq n}  \\
&= ( t_{m,*}^\kappa \circ \iota^{[ \kappa, \mu ]}_{K'_{n',m},u,*} ) (v_{n'} \cup w^k_{n'}) \\ 
&= t_{m,*}^\kappa ( \iota_{K'_{n',m},u,*}(v_{n'}) \cup u_*w^{[\kappa, \mu]}_{n'}) \\ 
&= t_{m,*} (t_{m, \flat}^\kappa( \iota_{K'_{n',m},u,*}(v_{n'})) \cup (t_{m})^*w^{[\kappa, \mu]}_{n}) \\ 
&=  (t_{m,*}^\kappa \circ  \iota_{K'_{n',m},u,*} ) v_{n'} \cup w^{[\kappa, \mu]}_{n} \\ 
&= (\operatorname{pr}^{K_{s,m}}_{K_{n,m}})_* ((t_{m,*}^\kappa \circ  \iota_{K'_{s',m},u,*} ) v_{s'} \cup w^{[\kappa, \mu]}_{s} )_{s \geq n} \\ 
&= (\mathrm{mom}^{k}_{{\G}, n} \circ t_{\infty ,m,*} \circ \iota_{\infty,u, *})( \mathbf{v}),
\end{align*} 
where the third equality follows since $u_*=(u^{-1})^*$ distributes over cup products, while the second, fifth and sixth ones follow by the projection formula. This shows the claimed commutativity.

Finally, we define \[ _c \mathbf{z}_m  : = (t_{\infty ,m,*} \circ \iota_{\infty,u, *})(_c \mathbf{z}_{\H,m}). \]
Its interpolation property follows by the commutativity of the diagram and Theorem \ref{gl2interpolation}: 
\begin{align*}  \mathrm{mom}^{k}_{{\G}, n}({} _c \mathbf{z}_m) &= (t_{m,*}^\kappa \circ \iota^{[ \kappa, \mu ]}_{K'_{n',m},u,*} \circ \mathrm{mom}^k_{{\H}, n'})(_c \mathbf{z}_{\H,m}) \\
&=( t_{m,*}^\kappa \circ \iota^{[ \kappa, \mu ]}_{K'_{n',m},u,*})( pr_{ 1 , n', m }^* ({}\operatorname{_cEis}^k_{\et,n'}) )  \\
&= {} _cz^{[\kappa, \mu]}_{\mathrm{\acute e t}, n , m}.
\end{align*}
\end{proof}

\subsubsection{Construction of the universal element}\label{universalclass}

It is a natural question to ask whether we can $p$-adically interpolate a larger family of geometric classes, improving Proposition \ref{interpG}. In order to do so, we need to work with the level groups $K'_{n, m}$ defined before, vary the cyclotomic coordinate and consider a twisted form of the moment maps to make all constructions compatible, as it will be explained below.

\begin{definition}
For any $\lambda=(\lambda_1 \geq \lambda_2 \geq \lambda_3)$ and integers $i,j$, we define the groups
\[ H^i_{\rm Iw}(\mathrm{Sh}_{\G}(K'_{\infty}), \mathcal{W}^{\lambda}_\zp(j)) := \varprojlim_{n, m \geq 1} H^i_{\et}(\mathrm{Sh}_{\G}(K'_{n, m}), \mathcal{W}^{\lambda}_\zp(j)), \] where the inverse limit for $n$ is taken with respect to the natural trace maps and the inverse limit for $m$ is taken with respect to the map $\tilde{\eta}_{p,*}^\lambda$ of Theorem \ref{cyclnormrelation}.
\end{definition}

In order to define the moment maps analogous to those of Lemma \ref{momG}, we need to introduce some further notation. Together with the vectors $W, X, Y,$ and $Z$ introduced in \S \ref{explicitbl}, we define, for any $\lambda = (\lambda_1 \geq \lambda_2 \geq \lambda_3)$ and $\mu \in \mathcal{A}(\lambda)$ as in Lemma \ref{choiceofsomehighestwv}, the elements
\[  X' = 2 (e_1 \wedge e_2) - e_1 \wedge e_3 \in W^{(1 \geq 1 \geq 0)}_\Z, \;\;\; Y' = e_1 \wedge e_2 - e_1 \wedge e_3  \in W^{(1 \geq 1 \geq 0)}_\Z, \;\;\; Z' = 2 (e_1 \wedge e_2 \wedge e_3) \in W^{(1 \geq 1 \geq 1)}_\Z, \]
and set
\[ (w')^{[\lambda, \mu]} = W^{\mu_1 - \mu_2 - \lambda_3} \cdot (X')^{\lambda_2 - \lambda_3 - \mu_2} \cdot (Y')^{\mu_2} \cdot (Z')^{\lambda_3} \in W^\lambda_\Z. \] Denote, for any integer $t \geq 1$, by $(w')^{[\lambda, \mu]}_t$ its projection modulo $p^t$. It will turn out that $(w')^{[\lambda, \mu]}$ is the projection to the $S$-highest weight eigenspace of the image $w^{[\lambda, \mu]}$ in $W^\lambda_{\zp}$ of the element $v^{[\lambda, \mu]}$ defined in Lemma \ref{choiceofsomehighestwv}, where we recall that $S$ denotes the $1$-dimensional split torus $diag(x^3, x^2, x^2, x, x, 1)$, and this is of course the motivation for defining such a vector, as will be clear during the proof of Theorem \ref{thmunivclass} below. 

\begin{lemma}\label{momlevel'plus}
 Let $\lambda = (\lambda_1 \geq \lambda_2 \geq \lambda_3)$ and let $\mu \in \mathcal{A}(\lambda)$ as in Lemma \ref{choiceofsomehighestwv}. Then, for $m \geq 1$, $n\geq 3(m+1)$, and $r \in \Z$, there exist maps
 
 \begin{align*}
  \mathrm{mom}^{[\lambda, \mu, r]}_{\G, n, m}: H^i_{\rm Iw}(\mathrm{Sh}_\G(K'_{\infty}), \zp(4)) &\xrightarrow{\sim} \varprojlim_{s, t} H^i_\et(\mathrm{Sh}_\G(K'_{s, t}), (\Z / p^t)(4)) \\
  &\to \varprojlim_{s, t} H^i_\et(\mathrm{Sh}_\G(K'_{s, t}), (\Z / p^t \otimes \mathcal{W}^{\lambda}_{\Z / p^t})(4 - \lambda_2 - r) ) \\
  &\to \varprojlim_{s, t} H^i_\et(\mathrm{Sh}_\G(K'_{s, t}), \mathcal{W}^{\lambda}_{\Z / p^t}(4 - \lambda_2 - r) ) \\
  &\to \varprojlim_{t} H^i_\et(\mathrm{Sh}_\G(K'_{n, m}), \mathcal{W}^{\lambda}_{\Z / p^t}(4 - \lambda_2 - r)) \\
  &\xrightarrow{\sim} H^i_\et(\mathrm{Sh}_\G(K'_{n, m}), \mathcal{W}^{\lambda}_\zp(4 - \lambda_2 - r)), \\
 \end{align*}
where the first map is an isomorphism, the second map is obtained by taking the cup product against the \'etale section associated to $(w')^{[\lambda, \mu]}_t \otimes \zeta_t^{\otimes(-\lambda_2-r)} $ in $H^0_\et(\mathrm{Sh}_\G(K'_{s, t}), \mathcal{W}^{\lambda}_{\Z / p^t}(- \lambda_2 - r))$, the third one is the Cartan product, the fourth one is obtained by taking the projection to level $K'_{n, m}$, and the last one is an isomorphism again.
\end{lemma}

\begin{proof}

As in Lemma \ref{momG}, to show that the moment map is well defined, we only need to verify that the image of $(w')^{[\lambda, \mu]}_t \otimes \zeta_t^{\otimes(-\lambda_2 -r)}$ in $\mathcal{W}^{\lambda}_{\Z / p^t}(- \lambda_2 - r)$ is fixed by $K'_{s, t}$, where $s \geq 3(t+1)$, so that it gives a well defined element in $H^0_\et(\mathrm{Sh}_\G(K'_{s, t}), \mathcal{W}^{\lambda}_{\Z / p^t}(4 - \lambda_2 - r) )$. This is an immediate consequence of the fact that $K'_{s, t}$ is contained in the kernel $K_\G(p^t)$ of reduction modulo $p^t$ and hence acts trivially in the whole representation $\mathcal{W}^{\lambda}_{\Z / p^t}(- \lambda_2 -r)$.
\end{proof}

Let $e_{\mathrm{ord}}:= \lim_{k\to \infty} {\mathcal{U}_p'}^{k!}$ be the ordinary idempotent acting on $H^7_{\rm Iw}(\Sh_{\G}(K'_{\infty}), \zp(4))$.
We can now state our main result of this section concerning the existence of a class interpolating all previous constructions at finite levels. 
 
\begin{theorem} \label{thmunivclass}
There exists a class ${}_c \tilde{\mathbf{z}} \in  H^7_{\rm Iw}(\Sh_{\G}(K'_{\infty}), \zp(4))$ such that, for any $\lambda = (\lambda_1 \geq \lambda_2 \geq \lambda_3)$, $\mu \in \mathcal{A}(\lambda)$ as in Lemma \ref{choiceofsomehighestwv} and $n, m \in \N$, we have
\[ t_{m,*}^\lambda ( \mathrm{mom}^{[\lambda, \mu, 0]}_{{\G}, n, m}( {}_c \tilde{\mathbf{z}} ) )= \left(\tfrac{\sigma_p^{3}}{\mathcal{U}_p'}\right)^{m} \cdot e_{\mathrm{ord}} ( {}_c z^{[\lambda, \mu]}_{n, m}). \]
\end{theorem}
 
\begin{proof}
First, we define the class ${}_c \tilde{\mathbf{z}} \in  H^7_{\rm Iw}(\Sh_{\G}(K'_{\infty}), \zp(4))$. Let $ H^1_\et(\Sh_{\GL_2}(K_1(\infty)), \zp(1))^{\mathrm{bc}}$ be the sub-module of  $H^1_\et(\Sh_{\GL_2}(K_1(\infty)), \zp(1))$ of elements compatible under base-change. Then, let
\[ \alpha: H^1_\et(\Sh_{\GL_2}(K_1(\infty)), \zp(1))^{\mathrm{bc}} \to H^7_{\rm Iw}(\Sh_{\G}(K'_{\infty}), \zp(4)) \]
be the map $\varprojlim_{n, m} {\mathcal{U}_p'}^{-m} \cdot e_{\mathrm{ord}} \cdot \iota_{K'_{n, m}, u, *} \circ \mathrm{pr}_{1, n, m}^*$. Explicitly, we have
\[ \alpha \big( (z_{n} )_{n \geq 1} \big) = \big( {\mathcal{U}_p'}^{-m} \cdot e_{\mathrm{ord}}(  \iota_{K'_{n, m}, u , *} \circ \mathrm{pr}_{1, n, m}^*( z_{n} )) \big)_{n, m \geq 1} \]
for any $(z_{n})_{n \geq 1} \in H^1_\et(\Sh_{\GL_2}(K_1(\infty)), \zp(1))^{\mathrm{bc}}$. The fact that this map is well defined follows from (the proofs of) Proposition \ref{normrelationsn} (for the compatibility for varying $n$) and Theorem \ref{cyclnormrelation} (for the compatibility for varying $m$).

We then define
 \[ {}_c \tilde{\mathbf{z}} = \alpha(_c\mathbf{Eis}_{\GL_2}). \]
We now move to proving the interpolation properties of ${}_c \tilde{\mathbf{z}}$. Note that we have a diagram
 \[ \xymatrix{
 H^1_\et(\Sh_{\GL_2}(K_1(\infty)), \zp(1))^{\mathrm{bc}}  \ar[r]^-{\alpha}  \ar[d]_{\mathrm{mom}^k_{{\GL_2}, n}}  & H^7_{\rm Iw}(\Sh_{\G}(K'_{\infty}), \zp(4)) \ar[r]^-{\mathrm{mom}^{[\lambda, \mu , 0 ]}_{{\G}, n, m}}  &  H^7_\et(\Sh_{\G}(K_{n, m}'), \mathcal{W}^\lambda_\zp(4 -\lambda_2)) \ar[d]^-{t_{m,*}^\lambda}\\
H^1_\et(\Sh_{\GL_2}(K_1(n)), \mathcal{H}^{(k, 0 , 0)}_\zp(1)) \ar[r]^-{\beta }  & H^7_\et(\Sh_{\G}(K_{n, m}'), \mathcal{W}^\lambda_\zp(4  - \lambda_2)) \ar[r]^-{\gamma} & H^7_\et(\Sh_{\G}(K_{n, m}), \mathcal{W}^\lambda_{\zp}(4  - \lambda_2)) ,  }
 \] 
where $\beta =  \iota_{K'_{n, m}, u, *}^{[\lambda, \mu]} \circ \mathrm{pr}_{1, n, m}^*$ and $\gamma= \left(\tfrac{\sigma_p^{3}}{\mathcal{U}_p'}\right)^{m}\cdot e_{\mathrm{ord}} \cdot t_{m,*}^\lambda$, with $e_{\mathrm{ord}}$ denoting now the ordinary idempotent $\lim_{k\to \infty} {\mathcal{U}_p'}^{k!}$ acting on $H^7_\et(\Sh_{\G}(K_{n, m}), \mathcal{W}^\lambda_{\zp}(4  - \lambda_2))$. We claim it is commutative. Indeed, as in Proposition \ref{interpG}, after translating the result to a calculation for the algebraic representation corresponding to $\mathcal{W}^\lambda_\zp$, it follows from proving that \[ t_{m,\flat}^\lambda ( u_* w_m^{[\lambda, \mu ]} ) =  t_{m,\flat}^\lambda ( (w')_m^{[\lambda, \mu ]}),\] as sections of $t_m^*(\mathcal{W}^\lambda_{\Z/p^m})$. 
Recall, from Remark  \ref{onthenormoftm}, that the map $t_{m,\flat}^\lambda$ acts trivially on the $S$-highest weight space of $\mathcal{W}^\lambda_{\Z/p^m}$ and by a positive power of $p^m$ on all the other eigenspaces, thus $t_{m,\flat}^\lambda$ factors through the projection to the highest weight space relative to $S$. Notice that
\begin{align*}   
u^{-1} \cdot W&=W \\
u^{-1} \cdot X&=X - 2 (e_1 \wedge e_2) - e_1 \wedge e_3 +e_2 \wedge e_3 \\
u^{-1} \cdot Y&=Y - 2 (e_2 \wedge e_3) - e_1 \wedge e_3 +e_1 \wedge e_2 \\
u^{-1} \cdot Z&=Z - 2 (e_1 \wedge e_2 \wedge e_3).
\end{align*}
Considering the twisting when moving from $V_\Z^\lambda$ to $W_\Z^\lambda$ and the normalisation of Lemma \ref{lemmaforpushforward}, one easily checks that, for any $m \geq 1$,
\begin{align*}   
\eta_p^{-m} u^{-1} \cdot W &= W \in W^{(1 \geq 0 \geq 0)}_{\Z / p^m} \\
p^{-m} \eta_p^{-m} u^{-1} \cdot X &= 2 (e_1 \wedge e_2) - e_1 \wedge e_3 = X' \in W^{(1 \geq 1 \geq 0)}_{\Z / p^m} \\
p^{-m} \eta_p^{-m} u^{-1} \cdot Y &= e_1 \wedge e_3 + e_1 \wedge e_2 = Y' \in W^{(1 \geq 1 \geq 0)}_{\Z / p^m} \\
p^{-2m} \eta_p^{-m} u^{-1} \cdot Z &= 2 (e_1 \wedge e_2 \wedge e_3) = Z' \in W^{(1 \geq 1 \geq 1)}_{\Z/ p^m}.
\end{align*}
We deduce from the above that the image of $p^{-m (\lambda_2 + \lambda_3)} \eta_p^{-m}$ acting on $w^{[\lambda, \mu]}$ in $W_{\Z / p^m}^\lambda$ is given by
\[ p^{-m (\lambda_2 + \lambda_3)} \eta_p^{-m} \cdot w^{[\lambda, \mu]} =  W^{\mu_1 - \mu_2 - \lambda_3} \cdot (X')^{\lambda_2 - \lambda_3 - \mu_2} \cdot (Y')^{\mu_2} \cdot (Z')^{\lambda_3}=(w')_m^{[\lambda, \mu]} \in W_{\Z / p^m}^\lambda. \] 
This shows that $t_{m,\flat}^\lambda ( u_* w_m^{[\lambda, \mu ]} )=t_{m,\flat}^\lambda(w')_m^{[\lambda, \mu]}$, since $t_{m,\flat}^\lambda(w')_m^{[\lambda, \mu]}=(w')_m^{[\lambda, \mu]}$.
%The theorem is now an immediate consequence of the interpolation properties of the elements $_c\mathbf{Eis}_{\GL_2}$ of Theorem \ref{gl2interpolation} and of the definition of the zeta elements ${}_c \tilde{z}^{[\lambda, \mu]}_{n, m}$ (cf. Definition \ref{integraldefofclasses}).

The theorem is now a  consequence of the interpolation properties of the elements $_c\mathbf{Eis}_{\GL_2}$ of Theorem \ref{gl2interpolation} and of Theorem \ref{vernormrel}.  

\end{proof}
 
\begin{remark}
As one can see from the above proof, the vector $(w')^{[\lambda, \mu]}$ is precisely the projection of the vector $w^{[\lambda, \mu]}$ to the $S$-highest weight eigenspace of $W^\lambda_{\zp}$. Notice that the result above will remain true if one modifies the construction of the moment maps of Lemma \ref{momlevel'plus}, by adding any element in the complement of the $S$-highest weight eigenspace to $(w')^{[\lambda, \mu]}$. We do not know whether there is a natural choice for those test vectors.
\end{remark}

\section{Iwasawa theory}

We finish with an application to the construction of a compatible system of classes in the Galois cohomology of the Galois representation associated to certain cohomological cuspidal automorphic representations of $\GSp_6$.

\subsection{Mapping to Galois cohomology} \label{Galcoh}

Let $\pi = \pi_f \otimes \pi_\infty$ be a cuspidal automorphic representation of $\G(\A)$ of level $U$ such that $U$ is sufficiently small and satisfies $\nu(U) = \widehat{\Z}^\times$, $\pi_\infty$ is in the discrete series and $\pi$ appears in $H^6_{\et}(\Sh_{G, \overline{\Q}}(U), \mathcal{W}^\lambda_L(4+q))$ for some weight $\lambda$ and finite extension $L$ of $\qp$, and $q=\tfrac{k-|\lambda|}{2}$, for some $k \geq 0$ as in Lemma \ref{branching2b}.\\ 
Let $N \in \N$ be the smallest number such that $ K_\G(N) \subseteq U$ (recall that $K_\G(N)$ denotes the the principal congruence subgroup of level $N$), let $\mathcal{H}$ denote the Hecke algebra generated over $\Z$ by the standard Hecke operators for primes $\ell$ not diving $N$. Let $L$ be the $p$-adic completion of the smallest number field containing the Hecke eigenvalues of $\pi$ and denote by $\mathscr{O}_L$ its ring of integers, by $k_L$ its residue field, and let $\mathcal{H}_{\mathscr{O}_L} = \mathcal{H} \otimes_\Z \mathscr{O}_L$. Finally, we denote by $\mathfrak{m} \subseteq \mathcal{H}_{\mathscr{O}_L}$ the kernel of the character $\mathcal{H}_{\mathscr{O}_L} \to k_L$ defined by $\pi$.

The study of the localisation at the Hecke ideal $\mathfrak{m}$ of the cohomology of Siegel varities with integral coefficients has been carried in \cite{MokraneTilouine}. Their study relies on the existence of a Galois representation associated to $\pi$, which is now known to exists thanks to the recent work \cite{KretShin}. We will assume throughout the hypotheses \textbf{(GO)} (Galois ordinary), \textbf{(RLI)} (residually large image, i.e. \emph{non-Eisenstein-ness}) made in \cite[\S 1]{MokraneTilouine}, and the hypotheses \textbf{(St)} and \textbf{(spin-reg)} made in \cite{KretShin} (where the reader is referred for the appropriate definitions).

\begin{proposition} [{\cite[Theorem 1]{MokraneTilouine}}] \label{TilMok} If $p > 5$ and $p - 1 > |\lambda| + 6$ then
\[ H^\bullet_\et(\Sh_{\G, \overline{\Q}}(U), \mathcal{W}^\lambda_{\mathscr{O}_L}(\star))_\mathfrak{m} = H^6_\et(\Sh_{\G, \overline{\Q}}(U), \mathcal{W}^\lambda_{\mathscr{O}_L}(\star))_\mathfrak{m} \] is a free $\mathscr{O}_L$-module of finite rank.
\end{proposition}

Let now $V_\pi$ be the Galois representation associated to $\pi$, according to \cite{KretShin}. When $U = K_{n, 0}$ for some $n \in \N$, Proposition \ref{TilMok} and the Hochschild-Serre spectral sequence give a map
\begin{align*}
H^7_\et(\Sh_\G(K_{n, m}), \mathcal{W}^\lambda_{\mathscr{O}_L}(4+q)) \to& H^1(\Q(\zeta_{p^m}), H^6_\et(\Sh_{\G, \overline{\Q}}(K_{n, 0}), \mathcal{W}^\lambda_{\mathscr{O}_L}(4+q))_\mathfrak{m}) \\
\to& H^1(\Q(\zeta_{p^m}), T_\pi) \\
\to& H^1(\qp(\zeta_{p^m}), T_\pi),
\end{align*}
where $T_\pi$ denotes the $\mathscr{O}_L$-stable lattice in $V_\pi$ given by the $\pi_f$-isotypic component of the \'etale cohomology group with $\mathcal{W}^\lambda_{\mathscr{O}_L}$-coefficients. In the above composition, the second arrow follows from the definition of $V_\pi$ (\cite[\S 8]{KretShin}) and the last arrow is the restriction to the decomposition group.

\begin{definition} We let ${}_c z^{\pi}_{m}$ be the image of ${}_c z^{[\lambda, \mu]}_{n, m}$, for a $\mu = (k + j \geq j) \in \mathcal{A}(\lambda)$, in any of the groups appearing in the above composition.
\end{definition}

\subsection{Iwasawa cohomology} 

Let us start by recalling some general notions of $p$-adic Hodge theory. Let $L / \qp$ be a finite extension and denote by $\mathrm{Rep}_L \mathscr{G}_{\qp}$ the category of continuous $L$-linear representations of the absolute Galois group $\mathscr{G}_\qp = \mathrm{Gal}(\overline{\Q}_p / \qp)$ of $\qp$. Let $V \in \mathrm{Rep}_L \mathscr{G}_{\qp}$ and let $T \subseteq V$ be a $\mathscr{G}_\qp$-stable $\mathscr{O}_L$-lattice. We will assume through all this section that our Galois representations are crystalline with negative Hodge-Tate weights. Let
\[ H^1_{\rm Iw}(\qp, V) := \varprojlim_m H^1(\qp(\zeta_{p^m}), T) \otimes_{\mathscr{O}_L} L = H^1(\qp, \mathscr{O}_L[[\Gamma]][1/p] \otimes V) \]
denote the Iwasawa cohomology of the representation $V$, where the projective limit is taken with respect to the corestriction map, where $\Gamma = \mathrm{Gal}(\qp(\zeta_{p^\infty}) / \qp) \cong \zpe$, $\mathscr{O}_L[[\Gamma]]$ is the Iwasawa algebra of $\Gamma$ and where the second isomorphism follows from Shapiro's lemma.

Recall that the trace maps associated to the natural projections \[ H^7_\et(\Sh_\G(K_{n, m + 1}), \mathcal{W}^\lambda_{\mathscr{O}_L}(4+q)) \to H^7_\et(\Sh_\G(K_{n, m}), \mathcal{W}^\lambda_{\mathscr{O}_L}(4+q))\] correspond to the corestriction maps in Galois cohomology. 
Let $U = K_{n, 0}$, $\pi$, and $T_\pi \subset V_\pi$ be as in \S \ref{Galcoh}. In addition, suppose that $\pi$ is $\mathcal{U}'_p$-ordinary, in the sense that $\mathcal{U}'_p$ acts on $T_\pi$ as multiplication by a $p$-adic unit $\alpha$. Then, Theorem \ref{vernormrel} immediately gives the following.
 
\begin{proposition} Let ${}_c z^{\pi}_{m, \alpha} \in H^1(\Q(\zeta_{p^m}), T_\pi)$ be the class defined as $\left(\tfrac{\sigma_p^3}{\alpha} \right)^m \cdot {}_c z^{\pi}_{m}$.\\  For $m \geq 1$, we have \[ \mathrm{cores}^{\Q(\zeta_{p^{m+1}})}_{\Q(\zeta_{p^{m}})} ({}_c z^{\pi}_{m+1, \alpha} ) = {}_c z^{\pi}_{m, \alpha}.\]
In particular, the class ${}_c z^{\pi}_\alpha = (\mathrm{res}_p({}_c z^{\pi}_{m, \alpha}))_{m \geq 1}$ defines an element in $H^1_{\rm Iw}(\qp, V_\pi)$.
\end{proposition}

\subsection{The $p$-adic spin L-function} 

We finally recall that the work of Perrin-Riou \cite{PerrinRiou2} allows one to associate a $p$-adic L-function to an element of the Iwasawa cohomology. Before stating the result we need to introduce some more notations. Since we claim no new results in this section and for the sake of brevity, we will not aim to give any motivation and we will refer the reader to \cite{FonctLpadiques} for it, from where we also borrow most of the notations.

Let $V \in \mathrm{Rep}_L \mathscr{G}_{\qp}$ and $\Gamma$ be as above. From the interpretation of the Iwasawa cohomology in terms of cohomology classes taking values in $V$-valued measures, we easily get, for any $j \in \Z$ and a locally constant character $\eta: \Gamma \to L^\times$, specialisation maps
\[H^1_{\rm Iw}(\qp, V) \to H^1(\qp, V(\eta \chi^j)), \; \; z \mapsto \int_\Gamma \eta \chi^j \cdot z . \] Recall that we have Bloch-Kato's exponential and dual exponential maps
\[ \exp: \Dcris(V(\eta \chi^j)) \to H^1(\qp, V(\eta \chi^j)), \;\;\; \exp^*: H^1(\qp, V(\eta \chi^j)) \to \Dcris(V(\eta \chi^j)), \] where $\Dcris(-)$ denotes the functor associating to a $p$-adic local Galois representation its crystalline module. For $j \gg 0$ (resp. $j \ll 0$) the exponential (resp. dual exponential) map is an isomorphism of $L$-vector spaces and we will denote by
\[ \log: H^1(\qp, V(\eta \chi^j)) \to \Dcris(V(\eta \chi^j)) \] to be either $\exp^*$ (if $j < 0$) or $\exp^{-1}$ (if $j \gg 0$). We can then consider the composition of the specialisation maps and the logarithm map to obtain elements in the crystalline module of some twist of $V$.

Recall that the one dimensional representation $L(\eta \chi^j) = L \cdot e_{\eta, j}$ is crystalline and $\mathbf{e}_{\eta, j}^{\rm dR} := G(\eta) t^{-j} e_{\eta, j}$ is a basis for the module $\Dcris(L(\eta \chi^j))$, where $t \in \Bcris$ denotes Fontaine's `$2 i \pi$'. Let $\mathbf{e}_{\eta, j}^{\rm dR, \vee} := G(\eta)^{-1} t^{j} e_{\eta^{-1}, -j}$, which is a basis for $\Dcris(V(\eta^{-1} \chi^{-j}))$. Tensor product against $\mathbf{e}^{\rm dR, \vee}_{\eta, j}$ gives natural identification maps \[ \Dcris(V(\eta \chi^j)) \xrightarrow{\sim} \Dcris(V). \]  Putting all this together, we get natural maps
\[ H^1_{\rm Iw}(\qp, V) \to \Dcris(V) : z \mapsto \log(\int_\Gamma \eta^{-1} \chi^{-j} \cdot z) \otimes \mathbf{e}_{\eta^{-1}, -j}^{\rm dR, \vee}. \]

Finally, let $ \mathscr{D}(\zpe, L)$ denote the space of $L$-valued distributions on $\zpe$, i.e. the topological $L$-dual of the space of $L$-valued locally analytic functions on $\zpe$. Recall the following

\begin{proposition} [\cite{PerrinRiou2}]
 Let $V \in \mathrm{Rep}_L (\mathscr{G}_\qp)$ be a crystalline representation. There exists a map
 \[ \mathrm{Log}_V : H^1_{\rm Iw}(\qp, V) \to \mathscr{D}(\zpe, L) \otimes_L \mathbf{D}_{\rm crys}(V) \]
 such that, for any $z \in H^1_{\rm Iw}(\qp, V)$, $j \in \Z$, $j \geq 0$ or $j \ll 0$, and $\eta : \zpe \to L^\times$ a finite order character of conductor $p^n$, $n \geq 0$, we have
 \[ \int_\zpe \eta x^j \cdot \mathrm{Log}_V(z) = e_p(n) \Gamma^*(j + 1) p^{m (j + 1)} \cdot \mathrm{log}( \int_\Gamma \eta^{-1} \chi^{-j} \cdot z) \otimes \mathbf{e}_{\eta^{-1}, -j}^{\rm dR, \vee}, \]
 where $e_p(n) = 1$ if $n \geq 1$ and $e_p(n) = (1 - p^{-1} \varphi^{-1})$ if $n = 0$ and where \[ \Gamma^*(j+1) = \left\{
  \begin{array}{c c}
    j! & \quad \text{ if $j \geq 0$} \\
    \frac{(-1)^{j-1}}{(-j - 1)!} & \quad \text{if $j < 0$}  \\
   \end{array} \right. \] is the leading coefficient of the Laurent series of the function $\Gamma(s)$ at $s = j + 1$.
\end{proposition}

The term $(1 - p^{-1} \varphi^{-1})$ for trivial $\eta$ is the usual Euler factor appearing in the interpolation properties of the $p$-adic L-functions. Recall that, when $V$ is the Galois representation associated to a modular form, the above theorem applied to Kato's Euler system gives, after projecting to a Frobenius eigenline, the usual $p$-adic $L$-funciton of the modular form as constructed, for instance, in \cite{MTT}. This should justify the following

\begin{definition}
 We define the $p$-adic spin L-function of $\pi$ to be
 \[ {}_c \mu^\pi_\alpha := \mathrm{Log}_V({}_c z^\pi_\alpha) \in \mathscr{D}(\zpe, \mathbf{D}_{\rm crys}(V_\pi)). \]
\end{definition}

Let us make some comments on our constructions. On the one hand, the relation between the special values of the $p$-adic spin L-function and the complex L-function are still mysterious. We expect an explicit reciprocity law to hold, relating values of Bloch-Kato's dual exponential maps of $\int_\Gamma \eta^{-1} \chi^{-j} \cdot {}_c z^\pi_\alpha$ to the values of the complex spin L-function ${\rm L}_{\rm Spin}(\pi \times \eta, j)$ for a certain range of integers $j$, as predicted by Perrin-Riou. This problem looks for the moment out of reach (at least for the authors). On the other hand, our motivic classes should be related to special values of the complex spin L-function, as predicted by the conjectures of Beilinson and Bloch-Kato, through the Beilinson regulator taking values in absolute Deligne cohomology.  \\ 
This fact and the commutativity between the syntomic and $p$-adic regulators via Bloch-Kato's exponential map should give a link between the special values of the $p$-adic spin L-function at sufficiently small negative integers $j$ and the complex values of the spin L-function.
It seems reasonable to adapt techniques developed so far in the literature (cf. \cite{kings98}, \cite{Gealy2}, \cite{lemmarf}, \cite{KLZ} etc.) to attack these problems and we expect to come back to these in a not so distant future.

\appendix

\section{Proof of Lemma \ref{cartdiag}}

We finally give a proof of Lemma \ref{cartdiag}. 
Let $u \in {\G}(\Af)$ be the element whose component at $p$ is $\left( \begin{smallmatrix} I & T \\ 0&I \end{smallmatrix} \right),\text{ for } T = \left( \begin{smallmatrix} 1&1&0 \\ 1&0&1 \\ 0&1&1 \end{smallmatrix} \right),$  and let $n,m \geq 1$ be such that $n\geq 3m+3$. Then,  the commutative diagram 
\begin{eqnarray}\label{diagraminproof} \xymatrix{  & & & \Sh_{\G}(K_{n,m+1}') \ar[d]^{pr}  \\ 
\Sh_{\H}(uK_{n,m+1}'u^{-1} \cap {\H})  \ar[urrr]^-{\iota^u_{K_{n,m+1}'}}  \ar[rrr]^-{pr \circ \iota^u_{K_{n,m+1}'}}  \ar[d]_{\pi_p} & & &  \Sh_{\G}(K_{n,m(p)}') \ar[d]^{\pi_p'} \\
\Sh_{\H}(uK_{n,m}'u^{-1} \cap {\H})  \ar[rrr]^-{ \iota^u_{K_{n,m}'}} & & &  \Sh_{\G}(K_{n,m}')}  \end{eqnarray}
has Cartesian bottom square.

 In order to show the Cartesianness of diagram \ref{diagraminproof}, it is enough to check that \begin{enumerate}
\item The map $pr \circ \iota^u_{K_{n,m+1}'}$ is a closed immersion or, equivalently, \[ u K_{n,m(p)}' u^{-1} \cap {\H} = u K_{n,m+1}' u^{-1} \cap {\H} ;\]
\item $[ u K_{n,m}' u^{-1} \cap {\H} :u K_{n,m+1}' u^{-1} \cap {\H}]=[K_{n,m}':K_{n,m(p)}']$.
\end{enumerate}
These two facts are shown in the next two lemmas.
\begin{lemma}
We have the equality of subgroups of ${\H}(\Af)$ \[ u K_{n,m(p)}' u^{-1} \cap {\H} = u K_{n,m+1}' u^{-1} \cap {\H}.\]
\end{lemma}
\begin{proof}
It suffices to show that if $g = {\matrix A B C D} \in K_{n,m(p)}'$ is such that $ug u^{-1} \in {\H}$ then $g \in K_{n,m+1}'$, i.e that $g \equiv I \text{ mod } p^{m+1}$. Writing down the condition $ug u^{-1} \in {\H}$, we get that $g$ is of the form
\[ g = \left( \begin{smallmatrix} a_1 & - c_2 & - c_3 & (a_1 - d_3) - c_2 & (a_1 - d_2) - c_3 & b_1 \\ - c_1 & a_2 & - c_3 & (a_2 - d_3) -  c_1 & b_2 & (a_2 - d_1) -  c_3 \\ - c_1 & - c_2 & a_3 & b_3 & (a_3 - d_2) -  c_1 & (a_3 - d_1) - c_2  \\  & & c_3 & d_3 &  c_3 &  c_3 \\ & c_2 & &  c_2 & d_2 &  c_2 \\ c_1 & &  & c_1 &  c_1 & d_1 \end{smallmatrix} \right). \] 
The congruences of the (1,2) and (1,3) entries give $c_2\equiv c_3 \equiv 0 \text{ mod } p^{m+1}$. Moreover, taking a look at the elements off the anti-diagonal of $B$, we easily deduce that $a_1 \equiv a_2 \equiv a_3 \equiv d_3 \equiv d_2 \equiv d_1 \equiv 1 \text{ mod }  p^{m+1}$.
\end{proof}
We are left with showing that the degrees of the two vertical maps of the bottom square of \eqref{diagraminproof} are equal. 
\begin{lemma}
We have \[[ u K_{n,m}' u^{-1} \cap {\H} :u K_{n,m+1}' u^{-1} \cap {\H}]=[K_{n,m}':K_{n,m(p)}']. \]
\end{lemma}

\begin{proof}
Since a system of coset representatives of $Q=K_{n,m}'/K_{n,m(p)}'$ determines one for \[ uK_{n,m}' u^{-1} /uK_{n,m(p)}'u^{-1},\] it suffices to prove that we can find $\{ \sigma_i \}_{i \in I}$ system of coset representatives for $Q$ such that $u \sigma_i u^{-1} \in {\H}$ for all $i \in I$. Consider the following set of elements of $K_{n,m}'$ whose conjugation by $u$ is in ${\H}$:
\[ \sigma_{\underline{v}} = \left( \begin{smallmatrix} 1+p^ma & -p^mr' & -p^mr & p^{2m}t^{'''} & p^{2m}t^{''} & p^{3m}k \\ & 1+p^{m}b & -p^{m}r & p^{m}c & p^{m}k' & p^{2m}t \\  & -p^m r' & 1+p^m d & p^m k^{''} & p^m e & p^{2m} t'  \\ & & p^m r & 1+p^m s & p^m r & p^m r \\ & p^m r' & & p^m r' & 1+p^m f & p^m r' \\  & &  &  & & 1 \end{smallmatrix} \right) , \]
where for each vector $\underline{v}\in \Z/p^3\Z \times (\Z/ p^2\Z)^{4} \times (\Z/ p\Z)^{5} =:V$ we consider one (and only one) lift 
$$(k,t,t',t'',t''',k',k'',r,r',s)\in {\hat{\Z}}^{10}$$ so that $\sigma_{\underline{v}}\in {\G}(\Af)$, where we have set
\[ a=r'+s+p^m t^{''}, \;\;\; b=r+p^m t, \] \[ \;\;\; c=r-s+p^m t, \;\;\; d=r'+p^m t', \] \[ e=r-s+p^m (t'+t^{''}-t^{'''}), \;\;\; f=r'-r+s+p^m (t^{'''}-t^{''}). \] We claim that $\{ \sigma_{\underline{v}}^{-1} \}_{{\underline{v}} \in V}$ (or a subset of it) is a system of coset representatives for the quotient $Q$. We only sketch the proof of this, which consists of a very long but straightforward calculation. Given $g = {\matrix A B C D} \in K_{n,m}'$, we wish to prove that there exists $\underline{v} \in V$ such that $\sigma_{\underline{v}} g= {\matrix E F G H } \in K_{n,m(p)}'$. 

Writing down carefully the eight equations modulo $p^{m+1}$, the four modulo $p^{2(m+1)}$ and the remaining one modulo $p^{3(m+1)}$, we determine $\underline{v}$ by choosing ten of those equations and showing, by the use of the symplectic equations, that the other three equations are redundant. Slightly more precisely, we have, after reducing the equations modulo $p^{m+1}$ 
%(note that, since $p \mid M$, we have $M^2 \equiv 0$ modulo $M p$)
\[
\begin{array}{l l}
{\left\{ 
	\begin{array} {l l} 
	a_{12} - p^{m} r' a_{22} - p^{m} r a_{32} \equiv 0 \;\; [p^{m+1}] \\
	a_{13} - p^{m} r' a_{23} - p^{m} r a_{33} \equiv 0 \;\; [p^{m+1}]
	\end{array}
	\right. } \\

{\left\{ 
	\begin{array} {l l} 
	d_{13} + p^{m} r \equiv 0 \;\; [p^{m+1}] \\
	d_{23} + p^{m} r' \equiv 0 \;\; [p^{m+1}]
	\end{array}
	\right.} \\

{\left\{ 
	\begin{array} {l l} 
	b_{21} + p^{m} (r - s) d_{11} + p^{m} k' d_{21} \equiv 0 \;\; [p^{m+1}] \\
	b_{22} + p^{m} (r - s) d_{12} + p^{m} k' d_{22} \equiv 0 \;\; [p^{m+1}] \\
	b_{31} + p^{m} k'' d_{11} + p^{m} (r - s) d_{21} \equiv 0 \;\; [p^{m+1}] \\
	b_{32} + p^{m} k'' d_{12} + p^{m} (r - s) d_{22} \equiv 0 \;\; [p^{m+1}] \\
	\end{array}
	\right. } \\
\end{array} \\
\]

From the second pair of equations we get $r$ and $r'$ and, after replacing $p^{m} r$ and $p^{m} r'$, the first pair becomes redundant by the use of the symplectic equations of $g$ \[ A^{t}I'_3D-C^{t}I'_3B=\nu(g)I'_3 .\] 

Indeed, comparing the entries $(2,3)$ gives \[a_{12}d_{33}+a_{22}d_{23}+a_{32}d_{13}-c_{12}b_{33}+c_{22}b_{23}+c_{32}b_{13}=0,  \] which reduces modulo $p^{m+1}$ to \[a_{12}+a_{22}d_{23}+a_{32}d_{13}\equiv 0 \; \; [p^{m+1}],\] which coincides with the first equation after substituting $d_{12}$ and $d_{23}$ with $-p^{m}r$ and $-p^{m}r'$. Similarly, we get the redundancy of the second equation by comparing the entries $(3,3)$ modulo $p^{m+1}$. 

  To solve $s, k$ and $k''$ from the third series of equations, one has to show that the rank of the matrix
\[ \left( \begin{smallmatrix} d_{11} & d_{21} & 0 & -b_{21} \\ d_{12} & d_{22} & 0 & -b_{22} \\ d_{21} & 0 & d_{11} & - b_{31} \\ d_{22} & 0 & d_{12} & -b_{32} \end{smallmatrix} \right) \] is three. The fact that its rank is at least three follows by the fact that the determinant of $A^{t}I'_3D$ is invertible modulo $p^{m+1}$ (all entries of $B$ are divisible by $p$) and so \[ det(D)\equiv d_{11}d_{22}-d_{21}d_{12}\equiv  d_{11}d_{22}  \; \; [p^{m+1}] \] is invertible as well. Hence, we can find a $3 \times 3$ minor with invertible determinant. Finally, the fact that the big determinant is zero follows from an application of the relation \[d_{12}b_{31}+d_{22}b_{21}\equiv d_{11}b_{32}+d_{21}b_{22} \; \; [p^{m+1}], \] from the symplectic equations of $g$ \[ B^{t}I'_3D-D^{t}I'_3B=0 .\] 
Indeed, unfolding the calculation of the determinant we get \begin{align*} d_{11}&\left[ d_{22}(d_{22}b_{21}-d_{11}b_{32})-d_{21}(d_{22}b_{22}-d_{12}b_{32}) \right ]-d_{12}\left[ d_{22}(d_{21}b_{21}-d_{11}b_{31})-d_{21}(d_{21}b_{22}-d_{12}b_{31}) \right ]\equiv  \\
 \equiv& d_{11}d_{22}(d_{22}b_{21}+d_{12}b_{31}-d_{11}b_{32}-d_{21}b_{22})+d_{12}d_{21}(d_{12}b_{32}+d_{21}b_{22}-d_{22}b_{21}-d_{12}b_{31})\equiv  \\ 
  \equiv& (d_{11}d_{22}-d_{12}d_{21})(d_{22}b_{21}+d_{12}b_{31}-d_{11}b_{32}-d_{21}b_{22}) \equiv 0 \; \; [p^{m+1}]
 \end{align*}

The rest of the equations follow more easily.
\end{proof}

\bibliographystyle{alpha}

\bibliography{normcompatiblesystemsofgaloiscohomologyclassesforGSp6}

\end{document}